\documentclass[11pt]{amsart}
\usepackage[english]{babel}
\usepackage{a4wide,color,graphicx,amsmath,amssymb,verbatim,mathrsfs,latexsym}
 \usepackage[colorlinks=true]{hyperref}
\allowdisplaybreaks
\newtheorem{theo}{Theorem}
\newtheorem{prop}{Proposition}
\newtheorem{lemma}{Lemma}

\newtheorem{cor}{Corollary}
\newtheorem{rem}{Remark}

\newcommand{\R}{\mathbb{R}}	
\newcommand{\N}{\mathbb{N}}	
\newcommand{\eps}{\varepsilon}	
\newcommand{\pa}{\partial}		
\newcommand{\Div}{\textrm{div}\,}	
\newcommand{\fun}[5]{  			
\begin{array}{cccc}
{#1}\,: & #2 & \to & #3 \\
\phantom{{#1}\,:\,} & #4 & \mapsto & #5
\end{array} }
\newcommand{\na}{\nabla}		
	
\def\dive{\textnormal{div}}

\newcommand{\IRd}{\int_{\R^d}}

\newcommand{\umlaut}{\"}

\newcommand{\Bt}{ {B^{(\tau)}} }
\newcommand{\mtau}{m^{(\tau)}}
\newcommand{\utau}{u^{(\tau)}}
\newcommand{\wtau}{w^{(\tau)}}
\newcommand{\Etau}{E^{(\tau)}}
\newcommand{\Rtau}{R^{(\tau)}}

\title{Global existence of weak even solutions for an isotropic Landau equation with Coulomb potential}
\author{Maria Gualdani and Nicola Zamponi}
\address{Department of Mathematics, George Washington University, 801 22nd Street NW, 20052 Washington DC, USA. \\
Institute for Analysis and Scientific Computing, Vienna University of  
	Technology, Wiedner Hauptstra\ss e 8--10, 1040 Wien, Austria}
\email{gualdani@gwu.edu, nicola.zamponi@tuwien.ac.at}
\date{\today}

\begin{document}

\thanks{MPG is supported by NSF DMS-1412748 and DMS-1514761. MPG would like to thank NCTS Mathematics Division Taipei for their kind hospitality. 
NZ acknowledges support from the Austrian Science Fund (FWF), grants P22108, P24304, W1245.} 

\begin{abstract}
%
%
%
In this manuscript we consider an isotropic modification for the Landau equation with Coulomb potential in three space dimensions.
Global in time existence of weak solutions for even initial data is shown by employing a time semi-discretization of the equation,
an entropy inequality and a uniform estimate for the second moment of the solution to the discretized problem.
Moreover, under an additional condition that has to be satisfied uniformly over time, uniform boundedness of the solution is proved,
with bounds depending solely on the mass, second moment and entropy of the solution.
A byproduct of our analysis is a proof of improved regularity for weak solutions to the Landau equation with Coulomb potential.
\end{abstract}

\maketitle

\pagestyle{headings}		

\markboth{Global existence of weak even solutions }{M. Gualdani, N. Zamponi}

\section{Introduction}

In recent years the nonlinear nonlocal parabolic equation   
\begin{equation}\label{landau}
\begin{cases} u_t = \Div(a[u]\na u - u\na a[u]),\qquad -\Delta a[u] = u,\qquad x\in \R^3,~~ t>0,\\
 u(\cdot,0) =u_0
\end{cases}	
  \end{equation}	
  has attracted the interest of several mathematicians since it possesses interesting mathematical properties. Formally (\ref{landau}) can be rewritten as 
  $$
  u_t = a[u] \Delta u + u^2,
  $$ 
 which is reminiscent to the semilinear heat equation and the Keller-Segel model. However the solutions to (\ref{landau}) behave fundamentally different. 
 
  Equation (\ref{landau}) can be seen as an isotropic modification of a model in plasma physics, the so-called Landau equation, which reads in the spatially homogeneous setting as  
  $$
  u_t = \Div(A[u]\na u - u\na a[u]),\qquad -\Delta a[u] = u,
  $$
  with the diffusion matrix $A[u]$ defined as
  $$
 A[u](x,t):= \frac{1}{8\pi} \int_{\mathbb{R}^3} \frac{1}{|y|}\left( \textrm{Id}- \frac{ y\otimes y}{|y|^{2}}\right)u(x-y)\;dy.
  $$
  Then (\ref{landau}) follows from the Landau equation if we replace $A[u]$ with its trace, since $\textrm{Tr}A[u] = a[u]$.
  
A modification of (\ref{landau}) was first proposed by Krieger and Strain in \cite{KS}: the authors show the global in time existence of radial and monotone decreasing solutions to the equation 
  $$
  u_t =  a[u] \Delta u + \alpha u^2,\quad \alpha \in (0,\frac{2}{3}).
  $$
  This result was later extended to $\alpha \in (0,\frac{74}{75})$ in \cite{GKS} by mean of a new non-local weighted Poincare inequality. 
  
  Recently the first author and collaborator showed in \cite{GG} global in time existence of smooth and bounded radial and monotone decreasing solutions to (\ref{landau}) for initial data that have finite mass, energy and entropy.  This last result puts in evidence how solutions to a non-linear equation with a non-local diffusivity such as $a[u]$ behave drastically different from (and better than) Keller-Segel or semilinear heat equation.
  
The condition of radial symmetry and monotone decreasing simplify the analysis in \cite{GG} considerably: bounds on the solution are obtained using comparison principle for (\ref{landau}) and for the associated mass function. Pointwise upper bounds are obtained using barriers: these arguments are based on the observation that certain functions are supersolutions for the elliptic operator under certain assumptions on the solution. This is where the radial symmetry and monotonicity come to play.   

Without the simplifying assumption of radial symmetry, the analysis of \eqref{landau} is challenging.
The nonlocal nature of the diffusion $a[u]$ prevents the equation from satisfying a comparison principle. 
Moreover maximum principle does not give useful insights, since we only know that at any point of maximum for $u$ it holds $u_t \le u^2$,
which does not provide us with a global-in-time upper bound for the solution.
We also remark that the techniques for the classical Landau equation do not apply directly to (\ref{landau}), as they
rely on the conservation of the second moment which is not the case here.  

In this manuscript we attempt to remove the condition of radial symmetry. 
These difficulties related to the lack of comparison/maximum principles 
are overcome by performing an exponential transformation $ u = e^w$  and rewriting (\ref{landau}) for $w$,
which allows us to show the non-negativity of the solutions.

Our first main result is concerned with the {\em{global existence of weak even solutions }}to (\ref{landau}) in dimension $d=3$.

From now on let $\gamma(x)\equiv (1+|x|)^{-1}$ for $x\in\R^{3}$, and {let us define the space $W^{1,\infty}_{c}(\R^{3})$ of compactly supported
$W^{1,\infty}(\R^{3})$ functions equipped with the following notion of convergence: a sequence $(\phi_{n})_{n\in\N}\subset W^{1,\infty}_{c}(\R^{3})$
converges in $W^{1,\infty}_{c}(\R^{3})$ to a function $\phi\in W^{1,\infty}_{c}(\R^{3})$ if (and only if) $\phi_{n}\to\phi$ in $W^{1,\infty}(\R^{3})$
and $\cup_{n\in\N}\mbox{supp}(\phi_{n})$ is bounded.}
\begin{theo}\label{thr.ex}
Let $u_0 : \R^3\to (0,\infty)$ be an even function such that $u_0\in L^1(\R^3)$, $\int_{\R^3}u_0\log u_0 dx < \infty$,
and $\int_{\R^{3}}|x|^{2}u_{0}(x)dx<\infty$.
Then there exists $u : \R^3\times [0,\infty)\to [0,\infty)$ even function such that, for every $T>0$,
\begin{align*}
& \sqrt{u} \in L^2(0,T; H^1(\R^3,\gamma(x)dx)),\qquad u, u\log u\in L^{\infty}(0,T; L^{1}(\R^{3})),\\
& a\in L^\infty(0,T; L^3_{loc}(\R^3)),\qquad \na a\in L^\infty(0,T; L^{3/2}_{loc}(\R^3)),\\
& {\pa_{t}u \in L^{r}(0,T; W^{2,\frac{r}{r-1}}(\R^{3},\gamma^{-\frac{1}{3(r-1)}}(x)dx)')\cap (L^{\infty}(0,T; W^{1,\infty}_{c}(\R^{3})))' }\quad\mbox{ for some }r>1,
\end{align*}
and $u$ satisfies the following weak formulation of \eqref{landau}
\begin{align}\label{landau.w}
&\int_0^T\langle \pa_{t}u\, , \phi\,\rangle dt + \int_0^T\int_{\R^3}(a\na u - u\na a)\cdot\na\phi\, dx dt = 0
\qquad\forall\phi\in L^{\infty}(0,T; W^{1,\infty}_{c}(\R^{3})),\\
& \lim_{t\to 0^+}u(t) = u_0\quad \mbox{in }W^{2,\frac{r}{r-1}}(\R^{3},\gamma^{-\frac{1}{3(r-1)}}(x)dx)'.\nonumber
\end{align}
The function $a[u]$ is given by 
\begin{align}
a[u](x,t) = \frac{1}{4\pi}\int_{\R^3}\frac{u(y,t)}{|x-y|}dy\qquad x\in\R^3,~~t>0. \label{a}
\end{align}
Moreover, the mass is conserved: $\int_{\R^{3}}u(x,t)dx = \int_{\R^{3}}u_{0}(x)dx$ for $t\in [0,T]$, while the second
moment of $u$ is locally bounded in time: $\sup_{t\in [0,T]}\int_{\R^{3}}|x|^{2}u(x,t)dx < \infty$.
\end{theo}
The proof is based on a time semi-discretization of (\ref{landau}).
{The discretized problem reads as
$$ \frac{u^{k}-u^{k-1}}{\tau} = \Div(a_{\tau}[u^{k}]\na u^{k} - u^{k}\na a_{\tau}[u^{k}]) - \tau \left( |w^{k}|^{2}w^{k} + \Div( |\na w^{k}|^{2}\na w^{k} )\right) , \quad k\geq 1, $$
where $\tau>0$ is the time-step, $a_{\tau}$ is given by \eqref{a.k}, and $u^{k} = \exp(w^{k})$. 
The new function $w^{k}$ is the so-called {\em entropy variable} \cite{Jue2015,Jue2016}, which relates
the entropy $H[u]=\int_{\R^{3}}u\log u\, dx$ and $u^{k}$ by the relation $w^{k} = \frac{\delta H}{\delta u}[u^{k}]$. In particular $w$ is the Frech\'et derivative of the entropy
functional with respect to $u$. 
The discretization of the problem and the regularizing term of $p-$Laplacian
type (with $p=4$) allow to obtain 
a bounded solution $w^{k}$, which yields in return a bounded and strictly positive function $u^{k}$.  }
At this point, the main difficulty is to obtain estimates for the solution to the semi-discretized problem that are uniform in terms of the time-step.
This is achieved via the mass conservation and the decay in time of the Boltzmann entropy functional. 

Suitable $L^2$ bounds for the gradient of the solution are derived from the entropy inequality. Here is where  the assumption of even initial datum comes to play. In fact, the entropy dissipation can be bounded as 
\begin{align*}
-\frac{dH}{dt} &=\frac{1}{2}\iint_{\R^{3}\times\R^{3}}\frac{u(x,t)u(y,t)}{|x-y|}|\na\log u(x,t) - \na\log u(y,t)|^{2} dx dy\\
&\geq \left(\int_{\R^{3}} \frac{u(x,t)}{1+|x|}dx \right)\left(\int_{\R^{3}}\frac{|\na u(x,t)|^2}{u(x)}\frac{dx}{1+|x|}\right) 
- \left| \int_{\R^{3}} \frac{\na u(x,t)}{1+|x|}dx \right|^2 ,
\end{align*}
and thanks to the assumption that $u$ is even, the last integral on the right-hand side vanishes.
The assumption of even initial data seems at the moment hard to remove. In \cite{Des2015} a lower bound for $-dH/dt$
in terms of the Fisher information was found using conservation of the second moment. In our case
a bound for the second moment is not a-priori given, but has to be proved; on the other hand lower bounds for the 
entropy dissipation are needed to control the energy. It seems therefore that the two problems, deriving a suitable entropy inequality
and proving the boundedness of the energy, are interwined. 
We overcome this problem by first deriving a lower bound for the entropy dissipation which involves the energy, then we bound the energy by means
of the entropy and its dissipation. The combination of these two bounds yields a discrete nonlinear equation for the energy, 
from which we deduce a power-like upper bound for the energy and subsequentially a uniform lower bound for the entropy dissipation.

Another problem in the analysis of \eqref{landau} is the fact that the only steady state of the system is trivial, i.e. it is identically zero.
As a consequence, one expects the entropy $H[u(t)]$ to approach $-\infty$ as $t\to\infty$. However, a lower bound for
the entropy is essential to control the Fisher information. We prove that the entropy $H[u]$
is greater than a (real) constant which depends on the second moment of $u$, thus obtaining the desired control on the Fisher information.

The gradient estimates derived from the entropy inequality and uniform in $\tau$ and will help in the time-continuous limit $\tau\to 0$. We obtain at first an ultra-weak solution to \eqref{landau}. 

Afterwards we perform higher regularity estimates for $a[u]$ and its gradient.
Those estimates are obtained through a novel technique which exploits the uniform boundedness of the entropy
and energy (instead of mass conservation). Subsequent applications of Jensen inequality lead to the estimate
$$ a^3[u] = \Phi(\xi(a[u])) \leq C\int_{|x-y|<\rho}\Phi\left(\frac{1}{|x-y|}\right)\xi(u(y,t))dy,\quad x\in\R^3,~~t\in [0,T], $$
where $R>0$ is arbitrary, $\rho>0$ depends on $R$ and the second moment of $u$, $\xi(s)=(1+s)\log(1+s)$ for $s\geq 0$, and
$\Phi : [0,\infty)\to [0,\infty)$ is such that $\Phi(\xi(s)) = s^{3}$ for $s\geq 0$. We point out that $\xi(u)$ is controlled by the modulus of the entropy
density $u |\log u|$, and therefore the $\xi(u)\in L^{\infty}(0,T; L^{1}(\R^{3}))$. 
On the other hand, by construction $\Phi(z)$ grows at infinity slower than $z^{3}$, which yields $\Phi(|\cdot|^{-1})\in L^{\infty}(0,T; L^{1}_{loc}(\R^{3}))$.
Since the right-hand side of the above estimate for $a^{3}[u]$ is a convolution of two integrable functions, we conclude that 
$a^{3}[u]\in L^{\infty}(0,T; L^{1}_{loc}(\R^{3}))$. In a similar fashion it is shown that $\na a[u]\in L^{\infty}(0,T; L_{loc}^{3/2}(\R^{3}))$.
These improved bounds yield the weak solution to \eqref{landau.w}. 

This last argument can be adapted to the case of the Landau equation, allowing an improved regularity result
for the (ultra-)weak solutions found in \cite[Corollary 1.1]{Des2015}.
\begin{cor}[Corollary of Theorem \ref{thr.ex}]\label{coro.ex}
Let $f=f(v,t)$ be a weak solution to the Landau equation with Coulomb potential as in \cite[Corollary 1.1]{Des2015}. 
Then $A[f]\in L^{\infty}(0,T; L^{3}_{loc}(\R^{3}))$, $\na a[f]\in L^{\infty}(0,T; L^{3/2}_{loc}(\R^{3}))$, and $f$ satisfies 
the following weak formulation 
\begin{align}\label{landau.true.w}
&\int_0^T\langle \pa_{t}f\, , \phi\,\rangle dt + \int_0^T\int_{\R^3}(A[f]\na f - f\na a[f])\cdot\na\phi\, dx dt = 0
\qquad\forall\phi\in L^{\infty}(0,T; W^{1,\infty}_{c}(\R^{3})).
\end{align}
\end{cor}
We stress the fact that Corollary \ref{coro.ex} holds for general initial data having finite mass, momentum, energy; the hypothesis
of even initial datum is not needed.

The second main result is a conditional regularity estimate for any solution to \eqref{landau} {{found in }} Thr.~\ref{thr.ex}. This result
is based upon the condition that $u$ and $a[u]$ satisfy a so-called {\em $\varepsilon$-Poincar\'e inequality}. 
We say that $u$ and $a[u]$ satisfy the $\varepsilon$-Poincar\'e inequality if given $\varepsilon>0$ as small as one wishes, 
there exists a constant $C_\varepsilon$ such that the following inequality holds true
 \begin{align}\label{eqn:epsilon_Poincare_inequality strongerII}
  \begin{array}{l}
	  \int_{\mathbb{R}^3}  u \phi^2\;dx \leq \varepsilon \int_{\mathbb{R}^3} a[u] |\nabla\phi|^2 \;dx + C_\varepsilon\int_{\mathbb{R}^3}	\phi^2\;dx
  \end{array}	  
\end{align} 
for any $\phi\in L^1_{loc}(\R^3)$ such that the right-hand side of \eqref{eqn:epsilon_Poincare_inequality strongerII} is convergent. The {\em $\varepsilon$-Poincar\'e inequality} was first introduced in \cite{GG17} and used to show regularization of solution to the original Landau equation. 

\begin{theo}[Conditional regularity]\label{thr.reg} Let $u$ be a solution to \eqref{landau} {{found in }} Thr.~\ref{thr.ex}. Assume $u$ is such that (\ref{eqn:epsilon_Poincare_inequality strongerII}) holds true. 
Then there exist constant $C=C(T,u_0,R)$ such that 
\begin{align*}
\|u\|_{L^\infty(B_R\times(t,T))} &\le C(T,u_0,R)\left(\frac{1}{t}+ 1\right)^{\alpha},\qquad t\in (0,T), \quad \alpha > 9/4,
\end{align*}
where $B_R\subset\R^3$ is any ball of radius $R$ and center at zero.
\end{theo}

Weighted Poincare's inequalities have been used to obtain informations about eigenvalues for the Schr\"odinger operators \cite{CWW, CW1, CW2, FP78, SW}. Let $L := -\Delta -w$, $w\ge 0$,  we say that $L$ is positive definite if an inequality of the form 
$$
\int w\phi^2 \;dx \le \int |\nabla \phi|^2 \;dx 
$$
holds true for any $\phi$. The above inequality is called the {\em uncertainty principle} \cite{F83}. 
Similarly one can look into eigenvalues of degenerate elliptic operators of the form  $L = -\textrm{div}(w_1 \nabla )-w_2$ with $w_1,w_2$ positive functions. 
In this case the uncertainty principle becomes a two-weight Poincare's inequality 
\begin{align*}
\int w_2\phi^2 \;dx \le \int w_1 |\nabla \phi|^2 \;dx .
\end{align*}
Therefore one can ask the question for which weights $\omega_1$ and $\omega_2$, inequalities of the form 
\begin{align}\label{weightedPS}
\int_Q \omega_1\phi^2 \;dx \le C \int_Q \omega_2 |\nabla \phi|^2 \;dx 
\end{align}
exist, for any function $\phi$ with either zero mean or compact support in $Q$. Several conditions that guarantee validity of (\ref{weightedPS}) have been given in the literature \cite{CW1, CW2, SW}. The most general one reads as 
$$
 |Q|^{\frac{2}{d}}\left (\frac{1}{|Q|}\int_{Q}\omega_1^r\;dv \right )^{\frac{1}{r}}\left (\frac{1}{|Q|}\int_{Q}\omega_2^{-r}\;dv \right )^{\frac{1}{r}}\le C
$$
for some $r>1$. The constant $C$ above appears on the right hand side of (\ref{weightedPS}).

Inspired by the similarity of (\ref{landau}) with the degenerate operator $L = -\textrm{div}(a[u] \nabla )-u$ the first author and collaborator proposed in \cite{GG17} the new inequality (\ref{eqn:epsilon_Poincare_inequality strongerII}). 

We outline the basic steps on how to get (\ref{eqn:epsilon_Poincare_inequality strongerII}) and refer to \cite{GG17} for more details. The main idea is the following: fix $\varepsilon>0$ and assume that there exists $R>0$ such that for each cube of length $R$ we have 
\begin{align}\label{small_P}
 |Q_R|^{\frac{2}{3}}\left (\frac{1}{|Q_R|}\int_{Q_R}u^r\;dx \right )^{\frac{1}{r}}\left (\frac{1}{|Q_R|}\int_{Q_R}a[u]^{-r}\;dx \right )^{\frac{1}{r}} \le  \varepsilon.
\end{align}
Then \eqref{small_P} and (\ref{weightedPS}) with $\omega_2 = a[u]$ and $\omega_1 = u$ imply:
$$
\int_{Q_R} u(\phi-(\phi)_{Q_R})^2 \;dx \le \varepsilon \int_{Q_R} a[u] |\nabla \phi|^2 \;dx,
$$
where $(\phi)_{Q_R}$ denotes the average of $\phi$ in the cube $Q_R$. A covering of $\R^3$ with cubes of size $R$ then yields (\ref{eqn:epsilon_Poincare_inequality strongerII}). {In the specific case of (\ref{landau}) there is another condition leading to (\ref{eqn:epsilon_Poincare_inequality strongerII}) that can replace (\ref{small_P}). Such condition involves a weighted estimate on $\nabla a[u]$ and the proof is summarized in Lemma~\ref{lem.cnd.epsPoi} in the Appendix. The main idea is the following: assume that for some $q>3$ 
\begin{align}
\gamma^{-1}\na a[u]\in L^{\infty}_{loc}(0,\infty; L^{q}(\R^{3})), \label{cnd.epsPoi}
\end{align}
then for each $\varepsilon >0$ (\ref{eqn:epsilon_Poincare_inequality strongerII}) holds.  Then we rewrite $ \int_{\mathbb{R}^3}  u \phi^2\;dx$  using the Poisson's equation $-\Delta a[u] = u$ and 
integrate by parts:
\begin{align*}
& \int_{\R^{3}}u\phi^{2}dx = -\int_{\R^{3}}\phi^{2}\Delta a \;dx = 2\int_{\R^{3}}\phi\na a\cdot\na \phi \; dx.
\end{align*}
Then inequality \eqref{eqn:epsilon_Poincare_inequality strongerII} follows from assumption \eqref{cnd.epsPoi} by means of 
a Gagliardo-Nirenberg inequality and the fact that $a[u](x)\geq C\gamma(x)$ for $x\in\R^{3}$, with $C$ being a positive
constant that depends only on the mass and second moment of $u$ (see Lemma~\ref{lem.cnd.epsPoi} in Appendix for all the details).} 

The fact that neither (\ref{small_P}) not (\ref{cnd.epsPoi}) can be proven at the moment is rather unsatisfactory. And consequently the results stated in Theorem \ref{thr.reg} should be viewed as conditional. However one can show that condition (\ref{small_P}) {\em nearly holds}: in fact it is easy to prove (see Proposition $2.14$ in \cite{GG17}) that there exists a constant $C$ only dependent on mass, first and second moment of $u$ such that
\begin{align*}
 |Q_R|^{\frac{2}{3}}\left (\frac{1}{|Q_R|}\int_{Q_R}u\;dx \right )\left (\frac{1}{|Q_R|}\int_{Q_R}a[u]\;dx \right )^{-1} \le C.
\end{align*}

The reason why the $\varepsilon$-Poincare's inequality (\ref{eqn:epsilon_Poincare_inequality strongerII}) helps controlling the quadratic nonlinearity in (\ref{landau}) is now easy to see: formal energy estimates yield
$$
\partial_t \int u^{p+1} \;dx  = -\frac{4p}{p+1} \int a[u] | \nabla u^{\frac{p+1}{2}} |^2 \;dx + p\int u^{p+2}\;dx.
$$
Inequality (\ref{eqn:epsilon_Poincare_inequality strongerII}) with $\phi = u^{\frac{p+1}{2}}$ and $\varepsilon \le \frac{3p}{p+1}$ implies 
$$
\partial_t \int u^{p+1} \;dx  +\frac{p}{p+1} \int a[u] | \nabla u^{\frac{p+1}{2}} |^2 \;dx \le  C(p)\int u^{p+1}\;dx.
$$
Theorem \ref{thr.reg} will be proven using a modification of this energy estimate and Moser's iteration. \\

The rest of the manuscript is organized as follows: Theorem ~\ref{thr.ex} and Corollary~\ref{coro.ex} are proven in Section \ref{section_theorem1} and  Theorem \ref{thr.reg} in Section \ref{section_theorem2}.


%
\section{Proof of Theorem \ref{thr.ex}.}\label{section_theorem1}
The proof is divided into several steps.\medskip\\
{\em Step 1: Construction of approximate solutions.}
Let $T>0$, $N\in\N$, $\tau=T/N$,
$\Bt\equiv\{x\in\R^{d}~:~ |x|<\tau^{-\alpha}\}$ for some $\alpha\in (0,1)$ to be specified later,
$w^{0} \equiv \log u_{0}$.

Let us consider the following time-discretized and regularized problem: for any $k = 1,\ldots,N$,
\begin{align}\label{landau.k}
& \mbox{find }w^{k}\in W^{1,4}(\Bt)\mbox{ such that: }\\
\nonumber
& \tau^{-1}\int_{\Bt}( \exp(w^{k}) - \exp(w^{k-1}) )\phi dx + \tau\int_{\Bt}( |w^{k}|^2 w^k\phi + |\na w^{k}|^2 \na w^k\cdot\na\phi )dx\\
\nonumber
&\qquad + \int_{\Bt}( a_{\tau}[\exp(w^{k})]\na\exp(w^{k}) - \exp(w^{k})\na a_{\tau}[\exp(w^{k})] )\cdot\na\phi dx = 0,\\
\nonumber
&\qquad\mbox{for all }\phi\in W^{1,4}(\Bt),
\end{align}
where the functional $a_{\tau}$ is defined as
\begin{align}
& a_{\tau}[u_{\tau}](x) = \int_{\Bt}\frac{u_{\tau}(y)}{{4\pi} |x-y|}dy\quad x\in\R^{3},\qquad \mbox{for any $u_{\tau}$ smooth enough.}
\label{a.k}
\end{align}
We solve \eqref{landau.k}, \eqref{a.k} by means of Leray-Schauder fixed point theorem. 
For given $z\in L^{\infty}(\Bt)$ and $\sigma\in [0,1]$ we first consider the following problem:
\begin{align}\label{landau.lin}
& \mbox{find }w\in W^{1,4}(\Bt)~\mbox{ such that }\quad A(w) = f,
\end{align}
where $A : W^{1,4}(\Bt)\to (W^{1,4}(\Bt))'$, $f \in (W^{1,4}(\Bt))'$ are defined as follows:
\begin{align*}
\langle A(w), \phi\rangle :=&\; \tau\int_\Bt (w^3\phi + |\na w|^2\na w\cdot\na\phi)dx + \int_\Bt a_\tau[\exp(z)]\exp(z)\na w\cdot\na\phi dx,\\
\langle f,\phi\rangle :=& -\sigma\tau^{-1}\int_\Bt( \exp(z) - \exp(w^{k-1}) )\phi dx + \sigma\int_\Bt\exp(z)\na a_{\tau}[\exp(z)] \cdot\na\phi dx,
\end{align*}
for all $\phi\in W^{1,4}(\Bt)$.

Let us verify that $f\in (W^{1,4}(\Bt))'$. Since for any $u\in L^{\infty}(\Bt)$, $x\in\Bt$ it holds
$$ |a_{\tau}[u](x)| + |\na a_{\tau}[u](x)|\leq C_{\tau}\int_{\Bt}\frac{|u(y)|}{|x-y|^{2}} dy\leq C_{\tau}\|u\|_{L^{\infty}(\Bt)}, $$
it follows that, for any $\phi\in W^{1,4}(\Bt)$,
\begin{align}\label{June29.bound.f}
|\langle f, \phi\rangle| \leq C(\tau,\|z\|_{L^\infty(\Bt)},w^{k-1})\|\phi\|_{W^{1,1}(\Bt)}\leq 
C(\tau,\|z\|_{L^\infty(\Bt)},w^{k-1})\|\phi\|_{W^{1,4}(\Bt)}.
\end{align}
Let us now show that $A$ is strictly monotone, coercive and semicontinuous. Given $w_1, w_2\in W^{1,4}(\Bt)$, let us consider:
\begin{align*}
&\langle A(w_1) - A(w_2), w_1 - w_2\rangle \\
&= \tau\int_\Bt ( (w_1^3 - w_2^3) (w_1 - w_2) + ( |\na w_1|^2\na w_1 - |\na w_2|^2\na w_2 ) \cdot\na(w_1-w_2) )dx\\
&\qquad + \int_\Bt a_\tau[\exp(z)]\exp(z)|\na(w_1-w_2)|^2 dx\\
&=\tau\int_\Bt ( w_1^2 + w_1 w_2 + w_2^2 )(w_1-w_2)^2 dx \\
&\qquad +\frac{\tau}{2}\int_\Bt (|\na w_1|^2 - |\na w_2|^2)\na(w_1+w_2)\cdot\na(w_1-w_2)dx\\
&\qquad +\frac{\tau}{2}\int_\Bt (|\na w_1|^2 + |\na w_2|^2)\na(w_1-w_2)\cdot\na(w_1-w_2)dx\\
&\qquad + \int_\Bt a_\tau[\exp(z)]\exp(z)|\na(w_1-w_2)|^2 dx\\
&=\tau\int_\Bt ( w_1^2 + w_1 w_2 + w_2^2 )(w_1-w_2)^2 dx + \frac{\tau}{2}\int_\Bt (|\na w_1|^2 - |\na w_2|^2)^2 dx\\
&\qquad +\frac{\tau}{2}\int_\Bt (|\na w_1|^2 + |\na w_2|^2)|\na(w_1-w_2)|^2 dx\\
&\qquad + \int_\Bt a_\tau[\exp(z)]\exp(z)|\na(w_1-w_2)|^2 dx \geq 0.
\end{align*}
This means that $A$ is monotone. Let us now assume that $\langle A(w_1) - A(w_2), w_1 - w_2\rangle = 0$. From the above computation it follows that
$$ \int_\Bt ( w_1^2 + w_1 w_2 + w_2^2 )(w_1-w_2)^2 dx = 0 $$
which implies that, for a.e. $x\in\Bt$, either $w_1(x)^2 + w_1(x) w_2(x) + w_2(x)^2 = 0$ or $(w_1(x)-w_2(x))^2=0$; in both cases $w_1(x)=w_2(x)$.
Therefore $A$ is strictly monotone.

Let us now consider, for a generic $w\in W^{1,4}(\Bt)$,
\begin{align}\label{June29.lb.A}
\langle A(w),w\rangle = \tau\|w\|_{W^{1,4}(\Bt)}^4 + \int_\Bt a_\tau[\exp(z)]\exp(z)|\na w|^2 dx . 
\end{align}
Clearly $ \|w\|_{W^{1,4}(\Bt)}^{-1} \langle A(w),w\rangle \to \infty$ as $\|w\|_{W^{1,4}(\Bt)}\to \infty$, i.e. $A$ is coercive.

It is straightforward to verify that, for any $w_1, w_2, w_3\in W^{1,4}(\Bt)$, the function 
$t\in [0,1]\mapsto \langle A(w_1+t w_2), w_3 \rangle $ is continuous. So $A$ is hemicontinuous.

From \cite[Thr.~26A]{Zei90} we deduce that \eqref{landau.lin} has a unique solution $w\in W^{1,4}(\Bt)$. 
Moreover from \eqref{June29.bound.f}, \eqref{June29.lb.A} it follows that
\begin{align}\label{June29.b.w}
\|w\|_{W^{1,4}(\Bt)}\leq C(\tau,\|z\|_{L^\infty(\Bt)},w^{k-1}).
\end{align}
Therefore we can define the operator 
$$ \fun{F}{L^{\infty}(\Bt)\times [0,1]}{L^{\infty}(\Bt)}{(z,\sigma)}{w}  $$
with $w\in W^{1,4}(\Bt)$ the unique solution to \eqref{landau.lin}.

We observe that $F(\cdot,0)\equiv 0$ (trivial). 
Additionally (\ref{June29.b.w}) implies that $F$ is compact, since F is bounded as an operator $ L^{\infty}(\Bt)\times [0,1]\to W^{1,4}(\Bt)$. 
We remind that the Sobolev embedding $W^{1,4}(\Bt)\hookrightarrow L^{\infty}(\Bt)$ is compact.
Standard arguments can be employed to prove that $F$ is also continuous.
Let us now consider $w\in L^{\infty}(\Bt)$ such that $F(w,\sigma)=w$ for some $\sigma$. Choosing $\phi=z=w$ in \eqref{landau.lin} leads to
\begin{align*}
& \sigma\tau^{-1}\int_{\Bt}( \exp(w) - \exp(w^{k-1}) )w dx + \tau \|w\|_{W^{1,4}(\Bt)}^{4} \\
&= - \sigma\int_{\Bt}( a_{\tau}[\exp(w)]\na\exp(w) - \exp(w)\na a_{\tau}[\exp(w)] )\cdot\na w dx \\
&= - \sigma\iint_{\Bt\times\Bt}\frac{e^{w(x)+w(y)}}{{4\pi} |x-y|}(\na w(x) - \na w(y))\cdot\na w(x) dx dy\\
&= - \frac{\sigma}{2}\iint_{\Bt\times\Bt}\frac{e^{w(x)+w(y)}}{{4\pi} |x-y|}|\na w(x) - \na w(y)|^{2} dx dy \leq 0.
\end{align*}
Moreover, since $u\in (0,\infty)\mapsto u\log u - u\in\R$ is convex, it follows that 
$$ (u_{1}\log u_{1} - u_{1}) - (u_{2}\log u_{2} - u_{2}) \leq \log(u_{1}) (u_{1}-u_{2})\qquad u_{1}, u_{2}>0. $$
Evaluating the above inequality for $u_{1}=\exp(w)$, $u_{2}=\exp(w^{k-1})$, it follows
$$ e^{w}(w-1) - e^{w^{k-1}}(w^{k-1}-1) \leq w (e^{w}-e^{w^{k-1}}). $$
Summarizing up
\begin{align}\label{landau.entr.fp}  
& \sigma\int_{\Bt}e^{w}(w-1) dx + \tau^{2}\|w\|_{W^{1,4}(\Bt)}^{4}\\ 
& + \frac{\tau\sigma}{2}\iint_{\Bt\times\Bt}\frac{e^{w(x)+w(y)}}{{4\pi} |x-y|}|\na w(x) - \na w(y)|^{2} dx dy
\leq \sigma\int_{\Bt}e^{w^{k-1}}(w^{k-1}-1) dx .\nonumber
\end{align}
{We point out that the first integral on the left-hand side of \eqref{landau.entr.fp} can be bound from below by a constant since
$e^s (s-1) + 1\geq 0$ for all $s\geq 0$ and $\Bt$ is bounded.} Therefore \eqref{landau.entr.fp}
provides us with a uniform (w.r.t. $\sigma\in [0,1]$) bound for $w$ in $W^{1,4}(\Bt)$, and therefore also in
$L^{\infty}(\Bt)$. Leray-Schauder's fixed point Theorem implies that $w^{k}\in L^{\infty}(\Bt)$ exists such that
$F(w^{k},1)=w^{k}$, i.e. $w^{k}\in W^{1,4}(\Bt)$ is a solution to \eqref{landau.k}, \eqref{a.k}. Furthermore, $w^{k}$ satisfies \eqref{landau.entr.fp}
with $\sigma=1$, i.e.
\begin{align}\label{landau.entr.k}
& \int_{\Bt} e^{w^{k}}(w^{k}-1) dx + \tau^{2}\|w^{k}\|_{W^{1,4}(\Bt)}^{4}\\ 
& + \frac{\tau}{2}\iint_{\Bt\times\Bt}\frac{e^{w^{k}(x)+w^{k}(y)}}{{4\pi} |x-y|}|\na w^{k}(x) - \na w^{k}(y)|^{2} dx dy
\leq \int_{\Bt}e^{w^{k-1}}(w^{k-1}-1) dx .\nonumber
\end{align}
We define now piecewise-constant in time functions which interpolate the sequences $e^{w^{k}}$, $w^{k}$. 
For $x\in\Bt$, $0\leq t\leq T$, let
\begin{align*}
\utau(x,t) &= u_{0}(x)\chi_{\{t=0\}} + \sum_{k=1}^{N}e^{w^{k}(x)}\chi_{t\in ((k-1)\tau,k\tau]},\\
\wtau(x,t) &= \log(u_{0}(x))\chi_{\{t=0\}} + \sum_{k=1}^{N}w^{k}(x)\chi_{t\in ((k-1)\tau,k\tau]}.
\end{align*}
We also define the discrete time derivative operator $D_{\tau}$ as
\begin{equation}
D_{\tau}f(t) = \tau^{-1}(f(t)-f(t-\tau)), \quad\tau\leq t\leq T,\quad\mbox{for any }f : [0,T]\to\R.\label{Dtau.def}
\end{equation}
With this new notation, \eqref{landau.k} can be rewritten as
\begin{align*}
& \int_{\Bt}(D_{\tau}\utau) \phi dx 
+ \tau\int_\Bt ( |\wtau|^2\wtau\phi + |\na\wtau|^2\na\wtau\cdot\na\phi )dx\\
&\qquad + \int_{\Bt}( a_{\tau}[\utau]\na\utau - \utau\na a_{\tau}[\utau])\cdot\na\phi dx = 0;
\end{align*}
integrating the above equality in the time interval $[0,T]$ yields
\begin{align}\label{landau.tau}
& \int_{0}^{T}\int_{\Bt}(D_{\tau}\utau) \phi dx dt 
+ \tau\int_0^T\int_\Bt ( |\wtau|^2\wtau\phi + |\na\wtau|^2\na\wtau\cdot\na\phi )dx dt\\
\nonumber
& + \int_{0}^{T}\int_{\Bt}( a_{\tau}[\utau]\na\utau - \utau\na a_{\tau}[\utau])\cdot\na\phi dx dt = 0,\quad
\phi\in L^{4}(0,T; W^{1,4}(\Bt)).
\end{align}
We point out that, from the previous computations,
\eqref{landau.tau} holds true for piecewise constant in time test functions $\phi$, but by a density argument we deduce that
the equation is fulfilled for any $\phi\in L^{4}(0,T; W^{1,4}(\Bt))$.

Ineq.~\eqref{landau.entr.k} can be rewritten in the new notation as
\begin{align}\label{landau.entr.tau}
& D_\tau H^{(\tau)}[\utau(t)] + \tau \|\wtau(t)\|_{W^{1,4}(\Bt)}^{4}\\ 
\nonumber
&\qquad + \frac{1}{2}\iint_{\Bt\times\Bt}\frac{\utau(x,t)\utau(y,t)}{|x-y|}|\na\wtau(x,t) - \na\wtau(y,t)|^{2} dx dy\leq 0,\\
& H^{(\tau)}[\utau] = \int_{\Bt} \utau(\log\utau - 1) dx .\nonumber
\end{align}

{\em Step 2: A-priori and uniform in $\tau$ estimates.} Here we derive some a-priori estimates, that will be employed to extract a convergent subsequence from $\utau$
and take the limit $\tau\to 0$ in \eqref{landau.tau}. We define for later convenience the following quantities:
\begin{align*}
\mtau(t) & =\int_{\Bt} \utau(x,t)dx,\qquad \Etau(t) = \int_{\Bt}|x|^2\utau(x,t) dx,\\
\Rtau(t) &= \min\left\{ \sqrt\frac{2\Etau(t)}{\mtau(t)} , \tau^{-\alpha}\right\} ,\qquad 0\leq t\leq T.
\end{align*}

{\em Uniform boundedness and positivity of the mass.}
We first observe that the mass $\mtau(t)$ is uniformely bounded and positive for $0\leq t\leq T$, $\tau>0$: let $t'\in [0,T]$ and choose $\phi(x,t) = \chi_{[0,t']}(t)$ in \eqref{landau.tau}.  
From (\ref{landau.entr.tau}) and H\umlaut{o}lder inequality it follows
\begin{align}\label{mass.b}
&  |\mtau(t') - \mtau(0)| \leq \tau\int_0^{t'}\int_{\Bt}|\wtau|^3 dx dt\leq C(T)(\tau |\Bt|)^{1/4} = C(T)\tau^{(1-\alpha)/4},\\
& \qquad 0\leq t'\leq T.\nonumber
\end{align}
Since $0<\alpha<1$ and $\mtau(0) = \int_{\Bt}u_0 dx$ is uniformely positive and uniformely bounded
w.r.t. $\tau>0$, it follows that positive constants $c_1$, $c_2$ exist such that
\begin{align}
c_1\leq\mtau(t)\leq c_2\qquad 0\leq t\leq T,~~\tau>0.\label{a_tau_ba_bb}
\end{align}

{\em Preliminary gradient estimate.}
We first find a lower bound for $\int_{|x|<R(t)}\utau(x,t)dx$, for $0\leq t\leq T$.
If $\Rtau(t)=\tau^{-\alpha}$ then $\int_{|x|<R(t)}\utau(x,t)dx = \mtau(t)\geq \frac{1}{2}\mtau(t)$.
On the other hand, if $\Rtau(t)<\tau^{-\alpha}$, then $\Rtau(t) = \sqrt{2\Etau(t)/\mtau(t)}$ and it holds
\begin{align*}
\int_{|x|<\Rtau(t)}\utau(x,t)dx &= \mtau(t) - \int_{\Rtau(t)<|x|<\tau^{-\alpha}}\utau(x,t)dx\\
\nonumber
&\geq \mtau(t) - \frac{1}{\Rtau(t)^2}\int_{\Rtau(t)<|x|<\tau^{-\alpha}}|x|^2\utau(x,t)dx\\
\nonumber
&\geq \mtau(t) - \frac{\Etau(t)}{\Rtau(t)^2} = \frac{1}{2}\mtau(t).
\end{align*}
Therefore, in any case,
\begin{align}
& \int_{|x|<\Rtau(t)}\utau(x,t)dx \geq \frac{1}{2}\mtau(t),\qquad t\in [0,T].
\label{lb.intBRu}
\end{align}

Let us consider the third term on the left-hand side of \eqref{landau.entr.tau}.
Since $|x-y|\leq |x| + |y| \leq (1+|x|)(1+|y|)$ for $x,y\in\R^3$, it follows
\begin{align*}
&\frac{1}{2}\iint_{\Bt\times\Bt}\frac{\utau(x,t)\utau(y,t)}{|x-y|}|\na\wtau(x,t) - \na\wtau(y,t)|^{2} dx dy\\
&\geq\frac{1}{2}\iint_{\Bt\times\Bt}\frac{\utau(x,t)\utau(y,t)}{(1+|x|)(1+|y|)}\left| \frac{\na \utau(x,t)}{\utau(x,t)} - \frac{\na \utau(y,t)}{\utau(y,t)} \right|^2 dx dy\\
&= \left(\int_{\Bt} \frac{\utau(x,t)}{1+|x|}dx \right)\left(\int_{\Bt}\frac{|\na \utau(x,t)|^2}{\utau(x)}\frac{dx}{1+|x|}\right) 
- \left| \int_{\Bt} \frac{\na \utau(x,t)}{1+|x|}dx \right|^2 .
\end{align*}
The assumption that $u_{0}$ is even implies that $\utau(\cdot,t)$ is even for $t>0$, and therefore $\pa_{x_{i}}\utau(\cdot,t)$ is odd for $t>0$, $i=1,2,3$.
In particular,
$$ \left| \int_{\Bt} \frac{\na \utau}{1+|x|}dx \right|^2 = \sum_{i=1}^3\left(\int_{\Bt}\frac{\pa \utau}{\pa x_i} \frac{dx}{1+|x|}\right)^2 = 0. $$
As a consequence
\begin{align*}
&\frac{1}{2}\iint_{\Bt\times\Bt}\frac{\utau(x,t)\utau(y,t)}{(1+|x|)(1+|y|)}\left| \frac{\na \utau(x,t)}{u(x,t)} - \frac{\na \utau(y,t)}{\utau(y,t)} \right|^2 dx dy\\
 &\geq \left(\int_{\Bt} \utau(x,t)\frac{dx}{1+|x|} \right)\left(\int_{\Bt}\frac{|\na\utau(x,t)|^2}{\utau(x)}\frac{dx}{1+|x|}\right) .
\end{align*}
We now wish to show a positive lower bound for $\int_{\Bt} \utau(x,t)\frac{dx}{1+|x|}$ for $0\leq t\leq T$. 
Remember that $\Rtau(t) = \min\{ (2\Etau(t)/\mtau(t))^{1/2} , \tau^{-\alpha}\}$. It holds
$$ \int_{\Bt}\frac{\utau(x,t)}{1+|x|}dx\geq 
\int_{|x|<\Rtau(t)}\frac{\utau(x,t)}{1+|x|} dx
\geq \frac{1}{1+R(t)}\int_{|x|<\Rtau(t)}\utau(x,t)dx . $$
From \eqref{lb.intBRu} and the fact that $\mtau(t)$ is uniformely positive and bounded it follows
\begin{align}\label{low.1}
\IRd \utau(x,t)\frac{dx}{1+|x|}\geq \frac{C}{(1+\Etau(t))^{1/2}}.
\end{align}
Therefore
\begin{align*}
&\frac{1}{2}\iint_{\Bt\times\Bt}\frac{\utau(x,t)\utau(y,t)}{(1+|x|)(1+|y|)}\left| \frac{\na \utau(x,t)}{u(x,t)} - \frac{\na \utau(y,t)}{\utau(y,t)} \right|^2 dx dy\\
 &\qquad\geq \frac{C}{(1+\Etau(t))^{1/2}}\int_{\Bt}\frac{|\na\utau(x,t)|^2}{\utau(x)}\frac{dx}{1+|x|},
\end{align*}
and the discrete entropy inequality leads to
\begin{align}\label{dHdt.est}
& D_\tau H^\tau[\utau(t)] + \tau\|\wtau\|_{W^{1,4}(\Bt)}^{4}
+ C\int_{\Bt}\frac{\left|\na\sqrt{\utau(x,t)}\right|^2}{(1+|x|)(1+\Etau(t))^{1/2}} dx \leq 0.
\end{align}

{\em Upper bound for $a$.} It holds
\begin{align}
&{4\pi} a_\tau[\utau](x,t)\\ 
\nonumber
&= \int_{|y|<\tau^{-\alpha},\, |x-y|<1}\frac{\utau(y,t)}{|x-y|}dy 
+ \int_{|y|<\tau^{-\alpha},\, |x-y|\geq 1}\frac{\utau(y,t)}{|x-y|}dy \equiv I_1 + I_2 .\label{a.dec}
\end{align}
The integral $I_2$ is uniformely bounded thanks to the boundedness of the mass.
To estimate $I_1$ we first use H\"older. Let $\eps\in (0,2)$.
\begin{align*}
I_1 &= \int_{|y|<\tau^{-\alpha},\, |x-y|<1}\frac{\utau(y,t)}{|x-y|}dy \\
&\leq \left( \int_{|y|<\tau^{-\alpha},\, |x-y|<1}\utau(y,t)^{(d-\eps)/(2-\eps)}dy \right)^{(2-\eps)/(d-\eps)}
\left( \int_{|x-y|<1}|x-y|^{-d+\eps}dy \right)^{(d-2)/(d-\eps)}\\
& \leq C\eps^{-1}\left( \int_{|y|<\min\{1+|x|,\tau^{-\alpha}\}}\utau(y,t)^{(d-\eps)/(2-\eps)}dy \right)^{(2-\eps)/(d-\eps)}\\
&= C\eps^{-1}\|\sqrt{\utau(t)}\|_{L^{2(3-\eps)/(2-\eps)}(B_{\rho(x)})}^2 ,
\end{align*}
with $\rho(x) = \min\{1+|x|,\tau^{-\alpha}\}$.
The interpolation inequality implies 
$$ \|\sqrt{\utau(t)}\|_{L^{2(3-\eps)/(2-\eps)}(B_{\rho(x)})} 
\leq \|\sqrt{\utau(t)}\|_{L^{2}(B_{\rho(x)})}^{1-\theta}\|\sqrt{\utau(t)}\|_{L^{6}(B_{\rho(x)})}^{\theta},\quad
\theta = \frac{3}{2}\frac{1}{3-\eps} . $$
Then, the Sobolev embedding $H^1\hookrightarrow L^{6}$ and the uniform boundedness of the mass implies
\begin{equation}
\|\sqrt{\utau(t)}\|_{L^{2(3-\eps)/(2-\eps)}(B_{\rho(x)})} \leq C(|x|)\|\sqrt{\utau(t)}\|_{H^1(B_{\rho(x)})}^\theta . \label{est.tmp}
\end{equation}
Notice that the constant $C$ in \eqref{est.tmp} depends on $|B_{\rho(x)}|$ and therefore on $|x|$. 
However, it is easy to estimate such constant (assuming w.l.o.g. that it is optimal). In fact, define
$$ C(R) \equiv \sup_{u\in H^1(B_R)\backslash\{0\}}\frac{\|u\|_{L^{6}(B_R)}}{\|u\|_{H^1(B_R)}} ,\qquad R\geq 1. $$
It is clear that each function $u\in H^1(B_R)$ can be written as $u(x) = v(x/R)$ with $v\in H^1(B_1)$. Moreover,
\begin{align*}
\|u\|_{L^{6}(B_R)}^{6} &= \int_{B_R}|v(x/R)|^{6} dx = R^3\int_{B_1}|v(y)|^{6} dy = R^3\|v\|_{L^{6}(B_1)}^{6},\\
\|u\|_{H^1(B_R)}^2 &= \int_{B_R}(|\na_x v(x/R)|^2 + |v(x/R)|^2 )dx 
= \int_{B_R}(R^{-2}|\na v(y)|^2 + |v(y)|^2 )\vert_{y=x/R}\, dx \\
&\geq R\int_{B_1}(|\na v|^2 + |v|^2 )dy = R\|v\|_{H^1(B_1)}^2 ,
\end{align*}
and so
$$ \frac{\|u\|_{L^{6}(B_R)}}{\|u\|_{H^1(B_R)}} \leq \frac{\|v\|_{L^{6}(B_1)}}{\|v\|_{H^1(B_1)}} ,\qquad R>1. $$
Thus \eqref{est.tmp} leads to
\begin{align}
\|\sqrt{\utau(t)}\|_{L^{2(3-\eps)/(2-\eps)}(B_{\rho(x)})} \leq C\|\sqrt{\utau(t)}\|_{H^1(B_{\rho(x)})}^\theta .\label{est.2}
\end{align}
From \eqref{est.2} and the boundedness of the mass we obtain
\begin{align*}
I_1 &\leq \eps^{-1}C \|\sqrt{\utau(t)}\|_{H^1(B_{\rho(x)})}^{2\theta }
\leq \eps^{-1}C(1 + \| \na \sqrt{\utau(t)}\|_{L^2(B_{\rho(x)})}^{2})^{\theta} \\
&\leq \eps^{-1}C\left(1 + (2+|x|)\int_{B_\rho(x)} \frac{|\na \sqrt{\utau(y,t)}|^2}{(1+|y|)} dy \right)^{\theta} \\
&\leq \eps^{-1}C(1+|x|)^{\theta}\left(1 + \int_\Bt \frac{|\na \sqrt{\utau(y,t)}|^2}{(1+|y|)} dy \right)^{\theta} .
\end{align*}
The estimates of $I_1$, $I_2$ imply
$$ a_\tau[\utau](x,t)^{1/\theta} \leq {C_\eps}(1+|x|)\left(1 + \int_\Bt\frac{|\na \sqrt{\utau(y,t)}|^2}{(1+|y|)} dy \right).$$
The discrete entropy inequality \eqref{dHdt.est} can be employed to bound the right-hand side of the above inequality:
$$ a_\tau[\utau](x,t)^{1/\theta} \leq {C_\eps}(1+|x|)\left(1 - (1+\Etau(t))^{1/2}D_\tau H^\tau[\utau(t)]\right) ,\quad \theta = \frac{3}{2}\frac{1}{3-\eps}. $$
We can restate the above estimate in a more handy way by defining $p = 1/\theta \in [1, 2)$:
\begin{align}\label{est.a.1}
a_\tau[\utau](x,t)^{p} &\leq C_p (1+|x|)\left(1 - (1+\Etau(t))^{1/2}D_\tau H^\tau[\utau(t)] \right) ,\quad 0\leq t\leq T,\\ 
&\qquad 1\leq p < 2.\nonumber
\end{align}

{\em Lower bound for $H[u]$.}
A lower bound for $H^\tau[u(t)]$ is here showed, which does not depend on $\tau$.
We point out that, since the integration domain converges, as $\tau\to 0$, to the whole space $\R^3$, this lower bound is not straightforward.
In fact, consider as an example a sequence $(u_n)_{n\in\N}\subset L^1(\R^3)$ such that $\|u_n\|_{L^1(\R^3)}=1$ for $n\geq 1$, and $u_n\to 0$ in $L^\infty(\R^3)$
as $n\to\infty$. It follows:
$$ \IRd u_n\log u_n dx \leq \IRd u_n\log\|u_n\|_{L^\infty(\R^3)} dx = \log\|u_n\|_{L^\infty(\R^3)}\to -\infty\quad\mbox{as }n\to\infty . $$
To prove this lower bound for $H[u]$, we write 
\begin{equation}
 H^\tau[u] =\int_{\{u<1\}} u(x)\log(u(x))dx + \int_{\{1\leq u < \tau^{-\alpha}\}} u(x)\log(u(x)) dx. \label{H.dec}
\end{equation}
We wish to show that the first integral is bounded from below by a suitable (real) constant. H\"older's inequality yields
\begin{align*}
&-\int_{\{u<1\}}u(x)\log u(x)\, dx = \int_{\{u<1\}}u(x)^{(1-\eps)/2} u(x)^{(1+\eps)/2}\log\frac{1}{u(x)}\, dx\\
&\leq \left( \int_{\{u<1\}}u(x)^{1-\eps}dx \right)^{1/2}\left( \int_{\{u<1\}}u(x)^{1+\eps} \left( \log\frac{1}{u(x)} \right)^2 dx \right)^{1/2}.
\end{align*}
Since the function 
$$s\in (0,1)\mapsto s^{\eps/2}\log(1/s)\in\R
$$ 
is bounded, we can estimate the term
$$
\int_{\{u<1\}}u(x)^{1+\eps} \left( \log\frac{1}{u(x)} \right)^2 dx
$$
with a constant that only depends on $\varepsilon$ and the $L^1$ norm of the initial datum. 
Therefore
\begin{align}
-\int_{\{u<1\}}u(x)\log u(x)\, dx &\leq C_\eps\left( \int_{\{u<1\}}u(x)^{1-\eps}dx \right)^{1/2}\leq C_\eps\left( \int_\Bt u(x)^{1-\eps}dx \right)^{1/2}.
\label{H.lb.1}
\end{align}
Let us now consider the integral
\begin{align*}
\int_\Bt u(x)^{1-\eps}dx &= \int_\Bt (1+|x|^2)^{1-\eps}u(x)^{1-\eps} (1+|x|^2)^{-(1-\eps)}  dx\\
&\leq\left(\int_\Bt (1+|x|^2)u(x)dx \right)^{1-\eps}\left(\int_\Bt (1+|x|^2)^{-(1-\eps)/\eps}dx\right)^{\eps}\\
&\leq\left(\int_\Bt (1+|x|^2)u(x)dx \right)^{1-\eps}\left(\IRd (1+|x|^2)^{-(1-\eps)/\eps}dx\right)^{\eps}
\end{align*}
For $\eps<2/5$ we obtain
\begin{align*}
\IRd u(x)^{1-\eps}dx &\leq C_\eps\left( \int_\Bt (1+|x|^2)u(x)dx \right)^{1-\eps}.
\end{align*}
From the above estimate, \eqref{H.dec} and \eqref{H.lb.1} we conclude
\begin{equation}
-H^\tau[\utau(t)]\leq -\int_{\{\utau<1\}}\utau(x,t)\log\utau(x,t)\, dx \leq C_\eps (1+E(t))^{(1-\eps)/2},\qquad t>0.\label{H.lb}
\end{equation}

{\em Boundedness of $\Etau$.}
We wish to find an upper bound for $\Etau$.
Let us choose $\phi(x,t) = |x|^2\psi(t)$ in \eqref{landau.tau} for some $\psi\in L^4(0,T)$, $\psi\geq 0$. It follows
\begin{align*}
&\int_0^T (D_\tau\Etau) \psi dt + \tau\int_0^T \int_\Bt ( |\wtau|^2\wtau |x|^2 + |\na\wtau|^2\na\wtau\cdot\na(|x|^2) )\psi dx dt\\
&+\int_0^T\int_{\Bt}x\cdot(a_\tau[\utau]\na\utau - \utau\na a_\tau[\utau])\psi dx dt = 0.
\end{align*}
H\umlaut{o}lder inequality yields
\begin{align*}
\tau\left| \int_{\Bt}|x|^2|\wtau|^2\wtau dx \right| &\leq \tau\|\wtau\|_{L^4(\Bt)}^3\left(\int_{\Bt}|x|^8 dx\right)^{1/4} 
\leq C\tau^{1/4 - (11/4)\alpha} ,\\
\tau\left| \int_{\Bt}x\cdot |\na\wtau|^2 \na\wtau dx \right| &\leq \tau\|\na\wtau\|_{L^4(\Bt)}^{3}\left( \int_{\Bt}|x|^4 dx \right)^{1/4} 
\leq C\tau^{1/4 - (7/4)\alpha}.
\end{align*}
Therefore 
$$ \tau\int_\Bt ( |\wtau|^2\wtau |x|^2 + |\na\wtau|^2\na\wtau\cdot\na(|x|^2) )dx\leq C
\qquad\mbox{if }\alpha \leq \frac{1}{11}. $$ 
Let us now consider
\begin{align*}
& -\int_{\Bt}x\cdot(a_\tau[\utau]\na\utau - \utau\na a_\tau[\utau])dx \\
& = \int_{\Bt}( -\Div(a_\tau[\utau]\utau x) + \utau\Div(a_\tau[\utau] x) + \utau x\cdot\na a_\tau[\utau])dx\\
& = -\int_{\pa\Bt}a_\tau[\utau]\utau |x|d\sigma + 3\int_\Bt a_\tau[\utau]\utau dx + 2\int_\Bt \utau x\cdot\na a_\tau[\utau] dx\\
&\leq 3\int_\Bt a_\tau[\utau]\utau dx + 2\int_\Bt \utau x\cdot\na a_\tau[\utau] dx .
\end{align*}
Let us consider the second integral on the right-hand side of the above inequality:
\begin{align*}
2\int_\Bt \utau x\cdot\na a_\tau[\utau] dx 
&= -\frac{2}{{4\pi}}\iint_{\Bt\times\Bt}\frac{x\cdot(x-y)}{|x-y|^3}\utau(x,t)\utau(y,t)dx dy\\
&= \frac{2}{{4\pi}}\iint_{\Bt\times\Bt}\frac{y\cdot(x-y)}{|x-y|^3}\utau(x,t)\utau(y,t)dx dy\\
&= -\frac{1}{{4\pi}}\iint_{\Bt\times\Bt}\frac{\utau(x,t)\utau(y,t)}{|x-y|}dx dy\\
&= -\int_\Bt a_\tau[\utau] \utau dx .
\end{align*}
Therefore
\begin{align*}
& -\int_{\Bt}x\cdot(a_\tau[\utau]\na\utau - \utau\na a_\tau[\utau])dx 
\leq 2\int_\Bt a_\tau[\utau]\utau dx .
\end{align*}
Summarizing up
$$ \int_0^T (D_\tau\Etau) \psi dt \leq \int_0^T\left(C + 2\int_\Bt a_\tau[\utau]\utau dx\right)\psi dt ,
\quad \psi\in L^4(0,T),~~\psi\geq 0, $$
which implies
\begin{align}\label{DEtau}
& D_\tau\Etau(t) \leq C + 2\int_\Bt a_\tau[\utau]\utau dx ,\qquad t\in [\tau,T].
\end{align}
Let $1\leq p < 2$. From H\umlaut{o}lder inequality and the boundedness of the mass it follows
($p' \equiv p/(p-1)$):
\begin{align}\label{est.Dtau.0}
D_\tau\Etau(t)&\leq C + 2\int_\Bt a_\tau[\utau](\utau)^{1/p} (\utau)^{1/p'}dx 
\leq C + C\left( \int_\Bt a_\tau[\utau]^p \utau dx \right)^{1/p}
\end{align}
Moreover, \eqref{est.a.1} implies
\begin{align}\label{critical}
&\left( \int_\Bt a_\tau[\utau]^p \utau dx \right)^{1/p} \\ \nonumber
&\leq C_p \left(1 - (1+\Etau(t))^{1/2}D_\tau H^\tau[\utau(t)] \right)^{1/p}\left( \int_\Bt (1+|x|) u(x,t) dx \right)^{1/p} .
\end{align}
From the above inequality it follows
\begin{align}\label{est.apu.0}
\left( \int_\Bt a_\tau[\utau]^p \utau dx \right)^{1/p} 
&\leq C_p \left(1 - (1+\Etau(t))^{1/2}D_\tau H^\tau[\utau(t)] \right)^{1/p}(1+\Etau(t)^{1/2})^{1/p}\\
\nonumber
&\leq C_p (1+\Etau(t))^{\frac{1}{p}}\left(1 - D_\tau H^\tau[\utau(t)] \right)^{1/p} .
\end{align}
Ineq.~\eqref{est.Dtau.0} and \eqref{est.apu.0} lead to
\begin{align*}
& D_\tau\Etau(t) \leq C_p (1+\Etau(t))^{\frac{1}{p}}\left(1 - D_\tau H^\tau[\utau(t)] \right)^{1/p} .
\end{align*}
It is more convenient to reframe the above inequality in the notation with $k$:
\begin{align}\label{mah.1}
E_k - E_{k-1} &\leq C_p\tau (1+E_k)^{\frac{1}{p}}\left(1 + \frac{H^\tau[u^{k-1}] - H^\tau[u^{k}] }{\tau} \right)^{1/p} ,\qquad k\geq 1.
\end{align}
We claim that from \eqref{mah.1} it follows
\begin{align}\label{mah.2}
(1+E_k)^{1-1/p} - (1+E_{k-1})^{1-1/p} &\leq C_p\tau \left(1 + \frac{H^\tau[u^{k-1}] - H^\tau[u^{k}] }{\tau} \right)^{1/p} ,\qquad k\geq 1.
\end{align}
In fact, if $E_k<E_{k-1}$, then \eqref{mah.2} is trivially true. On the other hand, if $E_k\geq E_{k-1}$, then \eqref{mah.1} implies
\begin{align*}
(1+E_k)^{1-1/p} &\leq \frac{1+E_{k-1}}{(1+E_k)^{\frac{1}{p}}} + C_p\tau \left(1 + \frac{H^\tau[u^{k-1}] - H^\tau[u^{k}] }{\tau} \right)^{1/p}\\
&\leq (1+E_{k-1})^{1-1/p} + C_p\tau \left(1 + \frac{H^\tau[u^{k-1}] - H^\tau[u^{k}] }{\tau} \right)^{1/p},
\end{align*}
that is, \eqref{mah.2} holds true.

Let us sum \eqref{mah.2} for $k=1,\ldots,\ell$ and apply a discrete H\umlaut{o}lder inequality (or just a convexity argument):
\begin{align*}
(1+E_\ell)^{1-1/p} &\leq (1+E_{0})^{1-1/p} + C_p\tau \sum_{k=1}^\ell\left(1 + \frac{H^\tau[u^{k-1}] - H^\tau[u^{k}] }{\tau} \right)^{1/p} \\
&\leq (1+E_{0})^{1-1/p} + C_p\tau\ell^{1-1/p}\left(\ell + \frac{H^\tau[u_0] - H^\tau[u^{\ell}] }{\tau} \right)^{1/p} .
\end{align*}
The sub-additivity of $x\in (0,\infty)\mapsto x^{1/p}$ leads to
\begin{align*}
(1+E_\ell)^{1-1/p} &\leq (1+E_{0})^{1-1/p} + C_p\tau\ell + C_p(\tau\ell)^{1-1/p}\left(H^\tau[u_0] - H^\tau[u^{\ell}] \right)^{1/p} .
\end{align*}
The uniform boundedness of $H^\tau[u_0]$ and \eqref{H.lb} imply
\begin{align*}
(1+E_\ell)^{1-1/p} &\leq (1+E_{0})^{1-1/p} + C_p\tau\ell + C_p(\tau\ell)^{1-1/p}\left(1 + (1+E_\ell)^{(1-\eps)/2}\right)^{1/p} \\
&\leq (1+E_{0})^{1-1/p} + C_p\tau\ell + C_p(\tau\ell)^{1-1/p}(1+E_\ell)^{(1-\eps)/2p}.
\end{align*}
Let us now choose $p=3/2$. As a consequence $1-1/p = 1/3 > (1-\eps)/3 = (1-\eps)/2p$, and therefore Young inequality leads to
\begin{align*}
(1+E_\ell)^{1/3} &\leq (1+E_{0})^{1/3} + C\tau\ell + C(\tau\ell)^{1/3}(1+E_\ell)^{(1-\eps)/3}\\
&\leq (1+E_{0})^{1/3} + C\tau\ell + C_\eps(\tau\ell)^{1/3\eps} + \frac{1}{2}(1+E_\ell)^{1/3},
\end{align*}
which, in the $\tau$ notation, it means
\begin{align*}
\frac{1}{2}(1+\Etau(t))^{1/3} &\leq (1+\Etau(0))^{1/3} + C t + C_\eps t^{1/3\eps} ,\qquad t>0,
\end{align*}
which means that, for some suitable constant $C_T>0$ (dependent on the final time $T>0$ but independent of $\tau$), it holds
\begin{equation}
\Etau(t)\leq C_T,\qquad t\in [0,T].\label{Etau.bound}
\end{equation}
{{\em Uniform bounds.}}
As a consequence of \eqref{Etau.bound}, inequalities \eqref{dHdt.est}, \eqref{H.lb}, \eqref{est.apu.0} become
\begin{align*}
& D_\tau H^\tau[\utau(t)] + \tau\|\wtau\|_{H^{2}(\Bt)}^{2}
+ C\int_{\Bt}\frac{\left|\na\sqrt{\utau(x,t)}\right|^2}{(1+|x|)} dx \leq 0,\\
&-H^{\tau}[\utau(t)]\leq {-\int_{\Bt\cap\{\utau<1\}}\utau(x,t)\log\utau(x,t)\, dx } \leq  C,\\
&\int_\Bt a_\tau[\utau]^p \utau dx \leq C_p (1 - D_\tau H^\tau[\utau(t)]),\qquad 1\leq p < 2,
\end{align*}
for $0\leq t\leq T$. The integration of the first and third inequalities with respect to time leads to
\begin{align}\label{dHdt.est.2}
& H^\tau[\utau(T)] + \tau\int_{0}^{T}\|\wtau\|_{H^{2}(\Bt)}^{2}dt
+ C\int_{0}^{T}\int_{\Bt}\frac{\left|\na\sqrt{\utau(x,t)}\right|^2}{(1+|x|)} dx dt \leq H^{\tau}[u_{0}],\\
\label{Htau.bound}
&-H^{\tau}[\utau(t)] \leq -\int_{\Bt\cap\{\utau<1\}}\utau(x,t)\log\utau(x,t) dx\leq C,\qquad 0\leq t\leq T,\\
\label{est.apu}
&\int_{0}^{T}\int_\Bt a_\tau[\utau]^p \utau dx \leq C_p,\qquad 1\leq p < 2.
\end{align}
Moreover (\ref{lb.intBRu}) implies 
\begin{align}
 \int_{|x|<\Rtau(t)}\utau(x,t)dx \geq \frac{c_1}{2},\qquad t\in [0,T],
\end{align}
and 
\begin{align}
a_\tau[\utau] & = \int_{\Bt}\frac{u_{\tau}(y)}{{4\pi} |x-y|}dy \ge \int_{|y|\leq\Rtau}\frac{u_{\tau}(y)}{{4\pi} |x-y|}dy \ge \frac{1}{\Rtau + |x|} \int_{|y|\leq\Rtau}\utau(x,t)dx \nonumber \\
& \ge \frac{c_1}{2(\Rtau + |x|)} \ge \frac{c_1}{2(C(T) + |x|)}, \label{a_tau_bound_below}
\end{align}
thanks to (\ref{Etau.bound}) and the fact that $m^{(\tau)}(t)$ is bounded above and below (\ref{a_tau_ba_bb}). 
{
We point out that from \eqref{dHdt.est.2}, \eqref{Htau.bound} it follows that 
\begin{equation}
\|\gamma \utau \|_{L^{1}(0,T; L^{3}(\Bt))}\leq C . \label{est.u.L1L3}
\end{equation}
To prove \eqref{est.u.L1L3} we start with the classical Sobolev inequality in three dimensions: 
 $$  \left(\int_{\Bt} g^6 \;dx \right)^\frac{1}{3} \le C \int_{\Bt} |\nabla g|^2 \;dx + {{\int_{\Bt} g^2\;dx }}, $$
 where the constant $C>0$ does not depend on $\tau$ (see proof of \eqref{est.2}) and apply it to 
 $$  g = \frac{\sqrt{\utau}}{(1+|x|)^{1/2}}.  $$
 Since 
 $$  |\nabla g | \le \frac{ |\nabla \sqrt{\utau} |}{(1+|x|)^{1/2}} +  \sqrt{\utau},  $$
 Sobolev inequality yields
  $$  \left(\int_{\Bt} \frac{(\utau)^3}{(1+|x|)^{3}}  \;dx \right)^{\frac{1}{3}} \le C \int_{\Bt}\frac{ |\nabla \sqrt{\utau} |^2}{(1+|x|)} +  \utau \;dx.  $$
 Integrating both sides in the time interval $(0,T)$ we get 
 \begin{align}\label{est_L1L3}
\int_0^T \left(\int_{\Bt} \frac{(\utau)^3}{(1+|x|)^{3}}  \;dx \right)^{\frac{1}{3}}dt 
& \le C\int_0^T \int_{\Bt}\frac{ |\nabla \sqrt{\utau} |^2}{(1+|x|)} \;dxdt +   \int_0^T \int_{\mathbb{R}^3}\utau \;dxdt  \nonumber \\
& \le C(T,u_0)
 \end{align}
 using mass conservation and estimate (\ref{est.nau}). \\
We wish now to find a set of estimates for $\utau$ which interpolate the bounds
$$ \sup_{t\in [0,T]}\Etau(t) = \| \gamma^{-2}\utau \|_{L^{\infty}(0,T; L^{1}(\Bt))} \leq C,\qquad
\|\gamma \utau \|_{L^{1}(0,T; L^{3}(\Bt))}\leq C . $$
Thanks to the bound $ \| \gamma^{-2}\utau \|_{L^{\infty}(0,T; L^{1}(\Bt))} \leq C$ it holds
\begin{align*}
& \int_{\Bt}\gamma^{3-5/p}(\utau)^{3-2/p}dx 
=\int_{\Bt}\left( \gamma^{-2}\utau \right)^{1/p}\left( \gamma^{3} (\utau)^{3} \right)^{1-1/p} dx\\
&\qquad \leq \left(\int_{\Bt}\gamma^{-2}\utau dx\right)^{1/p}\left(\int_{\Bt} \gamma^{3} (\utau)^{3} dx \right)^{1-1/p}
\leq C \left(\int_{\Bt} \gamma^{3} (\utau)^{3} dx \right)^{1-1/p} .
\end{align*}
Taking the power $\frac{p}{3(p-1)}$ of both members and integrating in $[0,T]$ leads to:
\begin{align*}
& \int_{0}^{T} \left( \int_{\Bt}\gamma^{3-5/p}(\utau)^{3-2/p}dx \right)^{\frac{p}{3(p-1)}}dt
\leq C \int_{0}^{T}\left(\int_{\Bt} \gamma^{3} (\utau)^{3} dx \right)^{1/3}dt . 
\end{align*}
By exploiting the bound $\|\gamma \utau \|_{L^{1}(0,T; L^{3}(\Bt))}\leq C$ we deduce:
\begin{align}\label{June20.bounds.u}
& \int_{0}^{T} \left( \int_{\Bt}\gamma^{3-5/p}(\utau)^{3-2/p}dx \right)^{\frac{p}{3(p-1)}}dt\leq C ,\qquad p>1.
\end{align}
In particular, choosing $p=3/2$ in \eqref{June20.bounds.u} yields
\begin{align}\label{June20.u53}
& \|\utau\|_{L^{5/3}(\Bt\times (0,T))}^{5/3}\leq \int_{0}^{T} \int_{\Bt}\gamma^{-1/3}(\utau)^{5/3}dx dt\leq C .
\end{align}
}
Now we will find a uniform bound for $D_{\tau}\utau$.
We define the functional space $X_{r}\equiv H^2(\R^3)\cap W^{2,r/(r-1)}(\R^{3},\gamma^{-1/3(r-1)}dx)$ for $r>1$.
Let us consider, for a given test function $\phi\in C^{\infty}_{c}(\Bt\times [0,T])$:
\begin{align*}
&\int_0^T\int_{\Bt}D_{\tau}\utau\, \phi dx dt + \tau\int_0^T\int_\Bt ( |\wtau|^2\wtau\phi + |\na\wtau|^2\na\wtau\cdot\na\phi )dx dt\\ 
&\qquad = -\int_0^T\int_{\Bt} \na\phi\cdot (a_{\tau}[\utau]\na\utau - \utau\na a_{\tau}[\utau])dx dt \\
&\qquad  = -\int_0^T\int_{\Bt}\na\phi\cdot(\na(a_{\tau}[\utau]\utau) - 2 \utau\na a_{\tau}[\utau])dx dt\\
&\qquad  = \int_0^T\int_{\Bt} a_{\tau}[\utau]\utau \Delta\phi dx dt + 2\int_{\Bt} \utau\na a_{\tau}[\utau]\cdot\na\phi dx dt .
\end{align*}
However, for $i=1,2,3$,
\begin{align*}
\utau\pa_{x_i}a_{\tau}[\utau] &= -\sum_{j=1}^3 \pa_{x_j x_j}^2 a_{\tau}[\utau]\, \pa_{x_i}a_{\tau}[\utau]\\
&= -\sum_{j=1}^3 \pa_{x_j}(\pa_{x_j} a_{\tau}[\utau]\, \pa_{x_i}a_{\tau}[\utau]) + \sum_{j=1}^3 \pa_{x_i x_j}^2 a\, \pa_{x_j}a \\
&= -\sum_{j=1}^3 \pa_{x_j}(\pa_{x_j} a_{\tau}[\utau]\, \pa_{x_i}a_{\tau}[\utau]) + \frac{1}{2}\pa_{x_i}(|\na a_{\tau}[\utau]|^2).
\end{align*}
Therefore
\begin{align}\label{weak.1}
&\int_0^T\int_{\Bt}D_{\tau}\utau\, \phi dx dt 
= -\tau\int_0^T\int_\Bt ( |\wtau|^2\wtau\phi + |\na\wtau|^2\na\wtau\cdot\na\phi )dx dt\\
\nonumber
&\qquad +\int_0^T\int_{\Bt} (a_{\tau}[\utau]\utau - |\na a_{\tau}[\utau]|^{2})\Delta\phi dx dt \\
\nonumber
&\qquad +2\int_{0}^{T}\int_{\Bt}\na a_{\tau}[\utau]\cdot (D_{x}^{2}\phi)\na a_{\tau}[\utau] dx dt, \qquad
\phi\in C^{\infty}_{c}(\Bt\times [0,T]).
\end{align}
It follows
\begin{align}\label{est.ut.1}
&\left|\int_0^T\int_{\Bt}D_{\tau}\utau\, \phi dx dt\right| 
\leq \tau\int_0^T\int_\Bt ( |\wtau|^3 |\phi| + |\na\wtau|^3 |\na\phi| )dx dt\\
\nonumber
&\qquad + C\int_0^T\int_{\Bt}|D_x^2\phi| (|\na a_{\tau}[\utau]|^2 + a_{\tau}[\utau]\utau)dx dt .
\end{align}
The contribution of the regularizing term is easily controlled by means of \eqref{dHdt.est.2}:
\begin{align}\label{ut.w}
&\tau\int_0^T\int_\Bt ( |\wtau|^3 |\phi| + |\na\wtau|^3 |\na\phi| )dx dt \\
\nonumber
&\qquad\leq \tau\|\wtau\|_{L^{4}(0,T; W^{1,4}(\Bt))}^3\|\phi\|_{L^{4}(0,T; W^{1,4}(\Bt))}\\
\nonumber
&\qquad\leq C\tau^{1/4}\|\phi\|_{L^{4}(0,T; W^{1,4}(\Bt))}.
\end{align}
Let us now consider:
\begin{align*}
&|\na a_{\tau}[\utau]|\leq \int_{\Bt}\frac{\utau(y,t)}{|x-y|^{2}}dy = f + g,\\
& f (x,t) \equiv \int_{\Bt\cap\{|x-y|<1\}}\frac{\utau(y,t)}{|x-y|^{2}}dy,\qquad g(x,t) \equiv \int_{\Bt\cap\{|x-y|\geq 1\}}\frac{\utau(y,t)}{|x-y|^{2}}dy.
\end{align*}
The function $g$ can be easily controlled by exploiting the boundedness of the mass:
\begin{align}\label{est.g}
& \|g\|_{L^{p}(\Bt\times(0,T))} \leq \left( \int_{|z|\geq 1}|z|^{-2p}dz \right)^{1/p}\|\utau\|_{L^{1}(\Bt\times(0,T))}\leq C,\qquad p>\frac{3}{2}.
\end{align}
{We recall that} $\gamma(x)=(1+|x|)^{-1}$, and let ${ \lambda\in\R }$ to be specified later. It holds
\begin{align*}
f(x,t)\gamma(x)^{(1-\lambda)/3} &= \int_{\Bt\cap\{|x-y|<1\}}(1+|x|)^{-(1-\lambda)/3}\frac{\utau(y,t)}{|x-y|^{2}}dy\\
&\leq C \int_{\Bt\cap\{|x-y|<1\}}(1+|y|)^{-(1-\lambda)/3}\frac{\utau(y,t)}{|x-y|^{2}}dy\\
& = C \int_{\Bt\cap\{|x-y|<1\}}\gamma(y)^{(1-\lambda)/3}\frac{\utau(y,t)}{|x-y|^{2}}dy
\end{align*}
and therefore
\begin{align*}
& \|f(t)^{2}\gamma^{2(1-\lambda)/3}\|_{L^{r}(\Bt)} = \|f(t)\gamma^{(1-\lambda)/3}\|_{L^{2r}(\Bt)}^{2} \\
&\leq \left( \int_{|s|<1} |s|^{-3+\eps}ds \right)^{4/(3-\eps)}\left( \int_{\Bt}\gamma(x)^{q(1-\lambda)/3}\utau(x,t)^{q}dx \right)^{2/q},
\end{align*}
with $1+\frac{1}{2r}=\frac{2}{3-\eps}+\frac{1}{q}$. 
{
Taking the power $r$ of both members in the above inequality and integrating it in time leads to
\begin{align*}
& \|f^{2}\gamma^{2(1-\lambda)/3}\|_{L^{r}(\Bt\times(0,T))}^{r} \leq 
C_{\eps}\int_{0}^{T}\left( \int_{\Bt}\gamma(x)^{q(1-\lambda)/3}\utau(x,t)^{q}dx \right)^{2r/q}dt .
\end{align*}
In order to control the right-hand side of the above inequality we apply \eqref{June20.bounds.u}. We wish to choose some $p\in (1,\infty)$
such that
$$ \frac{2r}{q} = \frac{p}{3(p-1)},\qquad \frac{q(1-\lambda)}{3} = 3 - \frac{5}{p},\qquad q = 3-\frac{2}{p} . $$
It follows that $ 2r = \frac{3p - 2}{3(p-1)} $. We want $r>1$ which is equivalent to $p<4/3$. 
Furthermore the relation $1+\frac{1}{2r}=\frac{2}{3-\eps}+\frac{1}{q}$ must be verified for some $\eps>0$, i.e. 
$\frac{1}{3}+\frac{1}{2r}-\frac{1}{q}>0$. This can be rewritten as
$$ \frac{1}{3}+\frac{3(p-1)}{3p - 2} - \frac{p}{3p-2}>0 $$
which is equivalent to $p>11/9$. Since $11/9<4/3$ we can choose $p\in (11/9,4/3)$ and conclude
}
%
%
\begin{align}\label{est.f}
& \|f^{2}\gamma^{1/3r}\|_{L^{r}(\Bt\times(0,T))} \leq C\qquad\mbox{for some }r>1.
\end{align}
From \eqref{est.g}, \eqref{est.f} and the facts that $|\na a_{\tau}[\utau]|\leq f + g$, $0\leq\gamma\leq 1$ we deduce that
\begin{align}
& \| |\na a_{\tau}[\utau]|^{2}\gamma^{1/3r} \|_{L^{r}(\Bt\times(0,T))}\leq C\qquad\mbox{for some }r>1.
\label{est.naa2}
\end{align}
{
Let us now consider, for $2 > p > r > 1$:
\begin{align*}
& \|a_{\tau}[\utau] \utau \|_{L^{r}(\Bt\times (0,T))} \\
&\qquad = \|a_{\tau}[\utau] (\utau)^{1/p} (\utau)^{1-1/p}\|_{L^{r}(\Bt\times (0,T))}\\
&\qquad \leq \|a_{\tau}[\utau] (\utau)^{1/p}\|_{L^{p}(\Bt\times (0,T))}\|(\utau)^{1-1/p}\|_{L^{pr/(p-r)}(\Bt\times (0,T))}\\
&\qquad = \|a_{\tau}[\utau]^{p}\utau\|_{L^{1}(\Bt\times (0,T))}^{1/p}\|\utau\|_{L^{(p-1)r/(p-r)}(\Bt\times (0,T))}^{1-1/p}
\end{align*}
From \eqref{est.apu} it follows
\begin{align*}
& \|a_{\tau}[\utau] \utau \|_{L^{r}(\Bt\times (0,T))} \leq C_{p}\|\utau\|_{L^{(p-1)r/(p-r)}(\Bt\times (0,T))}^{1-1/p}.
\end{align*}
We employ \eqref{June20.u53} to bound the right-hand side of the above inequality.
}
We wish to choose $p$, $r$ such that $2 > p > r > 1$ and $(p-1)r/(p-r) = 5/3$.
This implies $p = 2r/(5 - 3r)$. The constraint $p>r$ is automatically satisfied if $r>1$.
On the other hand, $p<2$ must hold, i.e. $r<5/4$.
Therefore we obtain
\begin{align}\label{a_unif_tau}
& \|a_{\tau}[\utau] \utau \gamma^{1/3r}\|_{L^{r}(\Bt\times (0,T))}\leq \|a_{\tau}[\utau] \utau \|_{L^{r}(\Bt\times (0,T))} \leq C,\qquad 1<r<\frac{5}{4}.
\end{align}
By putting the above estimate and \eqref{est.naa2} together we obtain:
\begin{align}\label{est.a.ugly}
& \| |\na a_{\tau}[\utau]|^{2}\gamma^{\frac{1}{3r}} \|_{L^{r}(\Bt\times(0,T))}
+ \|a_{\tau}[\utau] \utau \gamma^{\frac{1}{3r}}\|_{L^{r}(\Bt\times (0,T))} \leq C
\quad\mbox{for some }r > 1.
\end{align}
From \eqref{est.ut.1}, \eqref{ut.w} and \eqref{est.a.ugly} we conclude
\begin{align}\label{est.ut}
& \left|\int_0^T\int_{\Bt}D_{\tau}\utau\, \phi\, dx dt\right|\\
&\leq C\left( \tau^{1/4}\|\phi\|_{L^{4}(0,T; W^{1,4}(\Bt))} + \|\gamma^{-\frac{1}{3r}}\phi\|_{L^{\frac{r}{r-1}}(0,T; W^{2,\frac{r}{r-1}}(\Bt))} \right)
\quad\mbox{for some }r>1 .\nonumber
\end{align}
{\em Step 3: Limit $\tau\to 0$.} From \eqref{dHdt.est.2}, \eqref{Htau.bound}, \eqref{est.ut}, and the uniform boundedness of the mass it follows
\begin{align}
& \mbox{for any ball }B\subset\R^{3}\mbox{ there exists }C_{B}>0\mbox{ such that }\label{loc.b}\\
\nonumber
&\| (\utau)^{1/2} \|_{L^{2}(0,T; H^{1}(B))} + \|D_{\tau}\utau\|_{L^{r}(0,T; W^{2,r/(r-1)}(B)')}\leq C_{B} ,
\end{align}
where $r>1$ is as in \eqref{est.ut}.
Therefore we can apply Aubin-Lions Lemma in the version of \cite[Thr.~A.5]{Jue2016} and deduce that, 
for any $n\in\N$, a subsequence of $\utau$ exists, which is strongly convergent in $L^{1}(0,T; L^{3}(B_{n}))$,
where $B_{n} = \{x\in\R^{3}~:~|x|<n\}$. In particular, for any $n\in\N$, $\utau$ admits a subsequence, which is
a.e. convergent in $B_{n}\times [0,T]$. A Cantor diagonal argument allows us to extract from $\utau$ a subsequence,
which we will denote again with $\utau$, such that
\begin{align}
\label{utau.scnv}
& \utau\to u\qquad\mbox{strongly in }L^{1}(0,T; L^{3}(K))~~\mbox{for any compact set }K\subset\R^{3},\\
\label{utau.aecnv}
& \utau\to u\qquad\mbox{a.e. in }\R^{3}\times[0,T].
\end{align}
{\em Almost everywhere convergence of $a_{\tau}[\utau]$.}
Let us now show that $a_{\tau}[\utau]\to a[u]$ a.e. in $\R^{3}\times [0,T]$, where 
$$ a[u](x,t) \equiv \frac{1}{4\pi}\int_{\R^{3}}\frac{u(y,t)}{|x-y|}dy . $$
The uniform boundedness of mass:
\begin{align*}
& \|\utau\|_{L^{\infty}(0,T; L^{1}(\Bt))} = \sup_{[0,T]}\int_{\R^{3}}\utau\chi_{\Bt}dx\leq C ,
\end{align*}
and Fatou's Lemma imply that $u\in L^{\infty}(0,T; L^{1}(\R^{3}))$. Therefore
$$
\left| a_{\tau}[u](x,t) - a[u](x,t) \right| = \frac{1}{4\pi}\int_{|y|\geq\tau^{-\alpha}}\frac{u(y,t)}{|x-y|}dy
\leq \frac{\frac{1}{4\pi}}{\tau^{-\alpha}-|x|}\int_{|y|\geq\tau^{-\alpha}}u(y,t)dy\to 0
$$
as $\tau\to 0$, a.e. $x\in\R^{3}$, $t>0$. Therefore in order to prove that $a_{\tau}[\utau]\to a[u]$ a.e. $x\in\R^{3}$, $t>0$, it is sufficient to show that
$|a_{\tau}[\utau] - a_{\tau}[u]|\to 0$ a.e. $x\in\R^{3}$, $t>0$. Given any $\eps > \tau^{\alpha}$ it holds
\begin{align*}
& |a_{\tau}[\utau] - a_{\tau}[u]| \leq \frac{1}{{4\pi}}\int_{\Bt}\frac{|\utau(y,t)-u(y,t)|}{|x-y|}dy = \frac{1}{4\pi}(J_{1,\eps} + J_{2,\eps}),\\
& J_{1,\eps} \equiv \int_{|y|<1/\eps} \frac{|\utau(y,t)-u(y,t)|}{|x-y|}dy,\qquad 
J_{2,\eps} \equiv \int_{1/\eps\leq |y|<\tau^{-\alpha}} \frac{|\utau(y,t)-u(y,t)|}{|x-y|}dy.
\end{align*}
Let us consider $J_{1,\eps}$. H\"older inequality leads to
\begin{align*}
& J_{1,\eps}^{3} \leq \left(\int_{|y|<1/\eps} \frac{|\utau(y,t)-u(y,t)|^{3}}{|x-y|}dy\right)\left(\int_{|y|<1/\eps} \frac{dy}{|x-y|}\right)^{2}.
\end{align*}
By integrating the above inequality in an arbitrary compact set $K\subset\R^{3}$ we get
\begin{align*}
\int_{K}J_{1,\eps}^{3}dx &\leq \int_{|y|<1/\eps} \left(\int_{K}\frac{dx}{|x-y|}\right)|\utau(y,t)-u(y,t)|^{3} dy \,
\sup_{x\in K}\left(\int_{|y|<1/\eps} \frac{dy}{|x-y|}\right)^{2}\\
&\leq C_{\eps,K}\int_{|y|<1/\eps}|\utau(y,t)-u(y,t)|^{3} dy,
\end{align*}
which implies ($B_{1/\eps}\equiv\{x\in\R^{3}~:~|x|<1/\eps\}$):
\begin{align}
& \|J_{1,\eps}\|_{L^{1}(0,T; L^{3}(K))}\leq C_{\eps,K}\|\utau - u\|_{L^{1}(0,T; L^{3}(B_{1/\eps}))}\to 0\qquad\tau\to 0.\label{J1.cnv}
\end{align}
Let us now find an upper bound for $J_{2,\eps}$. For any $x\in\R^{3}$, $|x|<\eps^{-1}$, it holds
\begin{align*}
& |J_{2,\eps}|\leq \int_{1/\eps\leq |y|<\tau^{-\alpha}} \frac{|\utau(y,t)-u(y,t)|}{\eps^{-1}-|x|}dy
\leq \frac{\eps}{1-\eps |x|}\left( \int_{\Bt}\utau dx + \int_{\Bt}u dx\right).
\end{align*}
The uniform boundedness of the mass implies that, given any compact $K\subset B_{1/\eps}$, it holds
\begin{align}\label{J2.b}
& \|J_{2,\eps}\|_{L^{1}(0,T; L^{3}(K))}\leq C \eps \left( \int_{K}\frac{dx}{(1-\eps |x|)^{3}} \right)^{1/3} .
\end{align}
From \eqref{J1.cnv}, \eqref{J2.b} it follows that, given any compact $K\subset\R^{3}$,
\begin{align*}
& \limsup_{\tau\to 0}\| a_{\tau}[\utau] - a_{\tau}[u] \|_{L^{1}(0,T; L^{3}(K))}\leq C \eps \left( \int_{K}\frac{dx}{(1-\eps |x|)^{3}} \right)^{1/3} .
\end{align*}
Since the right-hand side of the above inequality tends to 0 as $\eps\to 0$, while the left-hand side is independent of $\eps$, we conclude that
the left-hand side vanishes. Therefore $\| a_{\tau}[\utau] - a_{\tau}[u] \|_{L^{1}(0,T; L^{3}(K))}\to 0$ as $\tau\to 0$, for any compact $K\subset\R^{3}$.
In particular, for any compact $K\subset\R^{3}$ there exists a subsequence of $a_{\tau}[\utau] - a_{\tau}[u]$ which is a.e. convergent in $K\times [0,T]$.
Choosing $K=B_{n}$, $n\in\N$, and applying again a Cantor diagonal argument, we extract a subsequence of $a_{\tau}[\utau] - a_{\tau}[u]$ such that
$a_{\tau}[\utau] - a_{\tau}[u]\to 0$ a.e. in $\R^{3}\times [0,T]$. Since we already know that $a_{\tau}[u]\to a[u]$ a.e. in $\R^{3}\times[0,T]$,
we conclude that, up to a subsequence,
\begin{align}\label{conv_punt_a}
a_{\tau}[\utau]\to a[u]\qquad\mbox{a.e. in }\R^{3}\times[0,T]. 
\end{align}
By exploiting the same strategy we can also show that, up to a subsequence,
$$ \na a_{\tau}[\utau]\to \na a[u]\qquad\mbox{a.e. in }\R^{3}\times[0,T]. $$
The a.e. convergence of $\utau$, $a_{\tau}[\utau]$, $\na a_{\tau}[\utau]$, and \eqref{est.a.ugly} allow us to apply Vitali's Theorem and deduce that,
for any compact $K\subset\R^{3}$, and for $i,j=1,2,3$,
\begin{align}
& a_{\tau}[\utau]\utau \to a[u] u\quad\mbox{strongly in }L^{1}(K\times [0,T]),\label{au.cnv}\\
& (\pa_{x_{i}}a_{\tau}[\utau])(\pa_{x_{j}}a_{\tau}[\utau]) \to (\pa_{x_{i}}a[u])(\pa_{x_{j}}a[u])\quad\mbox{strongly in }L^{1}(K\times [0,T]).\label{naa2.cnv}
\end{align}
From \eqref{est.ut}, \eqref{utau.scnv} it follows
\begin{align}\label{ut.wcnv}
\exists r>1~:~\forall K\subset\R^{3}\mbox{ compact, }\quad
& D_{\tau}\utau\rightharpoonup\pa_{t}u\quad\mbox{weakly in }L^{r}(0,T; W^{2,\frac{r}{r-1}}(K)').
\end{align}
Relations \eqref{ut.w}, \eqref{au.cnv}--\eqref{ut.wcnv} allow us to take the limit $\tau\to 0$ in \eqref{weak.1} and get
\begin{align}\label{landau.weak}
&\int_0^T\int_{\R^{3}}\pa_{t}u\, \phi dx dt = \int_0^T\int_{\R^{3}} (a[u]u - |\na a[u]|^{2})\Delta\phi dx dt \\
&\qquad +2\int_{0}^{T}\int_{\R^{3}}\na a[u]\cdot (D_{x}^{2}\phi)\na a[u] dx dt, \qquad
\phi\in C^{\infty}_{c}(\Bt\times [0,T]).\nonumber
\end{align}
Furthermore, \eqref{dHdt.est.2}, \eqref{est.a.ugly}, \eqref{est.ut} imply that, for some $r>1$,
\begin{align}\label{est.nau}
\sqrt{u} &\in L^2(0,T; H^1(\R^3,\gamma(x)dx)),\\
|\na a[u]|^{2},~ a[u] u &\in L^{r}(0,T; L^{r}(\R^{3},\gamma^{1/3}(x)dx)),\label{est.nau2}\\
\pa_{t}u &\in L^{r}(0,T; W^{2,r/(r-1)}(\R^{3},\gamma^{-1/3(r-1)}(x)dx)').\label{reg.ut}
\end{align}
{\em Mass conservation.}
We show now that the mass is conserved. Eq.~\eqref{mass.b} implies that 
\begin{equation}
\lim_{\tau\to 0}\int_{\Bt}\utau(x,t) dx = \|u_{0}\|_{L^{1}(\R^{3})}\qquad\mbox{a.e. }t\in [0,T].\label{mass.cnv}
\end{equation}
In particular, Fatou's Lemma implies that
$\|u(t)\|_{L^{1}(\R^{3})}\leq \|u_{0}\|_{L^{1}(\R^{3})}$, $t>0$. We want to show that equality holds, that is, the mass is conserved. Clearly 
$ \int_{\Bt}u(x,t) dx\to \int_{\R^{3}}u(x,t)dx $ as $\tau\to 0$. Therefore it is enough to show that $ \int_{\Bt}|\utau - u|dx\to 0$ as $\tau\to 0$. 
Cauchy-Schwartz inequality leads to
\begin{align*}
\int_{\Bt}|\utau - u|dx &= \int_{\Bt}\frac{|\utau - u|^{1/2}}{(1+|x|^{2})^{1/2}}|\utau - u|^{1/2}(1+|x|^{2})^{1/2}dx\\
&\leq \left( \int_{\Bt}\frac{|\utau - u|}{1+|x|^{2}} dx\right)^{1/2}  \left( \int_{\Bt}(1+|x|^{2}) |\utau - u| dx\right)^{1/2}.
\end{align*}
The uniform boundedness of mass and second moment $\Etau$ implies
\begin{align*}
& \int_{\Bt}|\utau - u|dx \leq C\left( \int_{\Bt}\frac{|\utau - u|}{1+|x|^{2}} dx\right)^{1/2} .
\end{align*}
Let $\eps>\tau^{\alpha}$, $B_{1/\eps}\equiv\{x\in\R^3~:~|x|<1/\eps\}$. It holds
\begin{align*}
& \|\utau - u\|_{L^2(0,T; L^1(\Bt))}^2 \leq 
C \int_0^T\int_{B_{1/\eps}}\frac{|\utau - u|}{1+|x|^{2}} dx + C\int_0^T\int_{\Bt\backslash B_{1/\eps}}\frac{|\utau - u|}{1+|x|^{2}} dx .
\end{align*}
From \eqref{utau.scnv} it follows that 
$$ \int_0^T\int_{B_{1/\eps}}\frac{|\utau - u|}{1+|x|^{2}} dx\to 0\qquad\tau\to 0, $$
while the uniform boundedness of the mass implies
$$ \int_0^T\int_{\Bt\backslash B_{1/\eps}}\frac{|\utau - u|}{1+|x|^{2}} dx\leq \frac{C}{1+\eps^{-2}}. $$
Therefore
\begin{align*}
&\limsup_{\tau\to 0}\|\utau - u\|_{L^2(0,T; L^1(\Bt))}^2 \leq  \frac{C}{1+\eps^{-2}}.
\end{align*}
Since the right-hand side of the above inequality tends to 0 as $\eps\to 0$ while the left-hand side is independent of $\eps$, we deduce that
$\lim_{\tau\to 0}\|\utau - u\|_{L^2(0,T; L^1(\Bt))}=0$, and therefore, up to a subsequence,
$\lim_{\tau\to 0}\|\utau - u\|_{L^1(\Bt)}=0$ a.e. $t\in [0,T]$. Since we knew already that 
$ \int_{\Bt}u(x,t) dx\to \int_{\R^{3}}u(x,t)dx $ as $\tau\to 0$, a.e. $t\in [0,T]$, we conclude that 
$ \int_\Bt\utau(x,t)dx\to \int_{\R^3} u(x,t) dx$ as $\tau\to 0$ a.e. $t\in [0,T]$. This fact and \eqref{mass.cnv} imply that
$ \int_{\R^3}u(x,t)dx = \int_{\R^3}u_0(x)dx$ a.e. $t\in [0,T]$, i.e. the mass is conserved.\medskip\\
{\em Weak formulation. }
We prove now that 
\begin{equation}
a[u]\in L^\infty(0,T; L^3_{loc}(\R^3)),~\na a[u]\in L^\infty(0,T; L^{3/2}_{loc}(\R^{3})). \label{a.reg.new}
\end{equation}
This will allow us to obtain the weak formulation of \eqref{landau} from \eqref{landau.weak}.

Let us define $\xi(s)=(1+s)\log(1+s)$, $s>0$. Since the mapping $s\in [0,\infty)\mapsto\xi(s)\in [0,\infty)$ is invertible,
we can define, for a given $p\geq 3/2$, the function $f_p : [0,\infty)\to [0,\infty)$ as $f_p(\xi(s)) = s^p$, $s>0$.
We point out that $f_p$ is convex for any $p\geq 3/2$. In fact, differentiating the relation $f_p(\xi(s)) = s^p$ once with respect to $s$ implies
$$ f_p'(\xi(s)) = \frac{p s^{p-1}}{1+\log(1+s)},\qquad s > 0, $$
and differentiating again leads to
$$ (1+\log(1+s))f_p''(\xi(s)) = ps^{p-2}\left( p -1 - \frac{s}{1+s}\frac{1}{(1+\log(1+s))^{2}} \right),\qquad s > 0. $$
It is easy to see that the function $s\mapsto (p-1)(1+s)(1+\log(1+s))^{2}-s$ is nondecreasing. Since it is positive at $s=0$, this means that 
it is positive for $s>0$. In particular $f_p''(\xi(s))\geq 0$ for $s>0$. Being $\xi$ invertible, we conclude that $f_p$ is convex.

Let $R>0$ arbitrary, let $m$ be the mass, and define
$$ \rho = 2\sqrt{R^2 + \frac{1}{m}\sup_{0\leq t\leq T}E(t)} < \infty . $$
We split ${4\pi} a = a_1 + a_2$ with 
$$ a_1(x,t) = \int_{|x-y|<\rho}\frac{u(y)}{|x-y|}dy,\qquad a_2(x,t) = \int_{|x-y|>\rho}\frac{u(y)}{|x-y|}dy. $$
We will show that $a_i\in L^\infty(0,T; L^3(B_R))$, $i=1,2$.

Let us first consider $a_2$. Since $1+|y|^2\geq C_R'(|x|^2+|y|^2)\geq C_R'' |x-y|^2$ for $|x|<R$, $y\in\R^3$, it holds
\begin{align*}
\int_{B_R}a_2(x,t)^3 dx \leq C_R\int_{B_R}\left(\int_{|x-y|>\rho}\frac{(1+|y|^2)u(y,t)}{|x-y|^3}dy\right)^3 dx.
\end{align*}
Jensen's inequality and the uniform boundedness of the second moment of $u$ imply
\begin{align*}
\int_{B_R}a_2(x,t)^3 dx \leq C_R\int_{B_R}\int_{|x-y|>\rho}\frac{(1+|y|^2)u(y,t)}{|x-y|^9}dy dx 
\leq C_R\int_{\R^3}(1+|y|^2)u(y,t)dy\leq C_R .
\end{align*}
Therefore $a_2\in L^\infty(0,T; L^3(B_R))$. Let us now focus our attention on $a_1$. 
The fact that $c_1 \equiv \int_{|x-y|<\rho}|x-y|^{-1}dx<\infty$ is independent of $x\in\R^3$, the convexity of $\xi$
and Jensen's inequality imply
$$ \xi(a_1(x,t)) \leq \frac{1}{c_1}\int_{|x-y|<\rho}|x-y|^{-1}\xi(c_1 u(y,t))dy\qquad x\in\R^3,~~t\in [0,T] . $$
Let us now notice that $\xi(c_1 u) \in L^\infty(0,T; L^1(\R^3))$. In fact, \eqref{dHdt.est.2}, \eqref{Htau.bound} imply that 
$$ \int_{\Bt}\utau(x,t)(\log\utau(x,t))_+ dx = H^{\tau}[\utau(t)] - \int_{\Bt\cap\{\utau<1\}}\utau(x,t)\log\utau(x,t) dx \leq C $$
for $0\leq t\leq T$. By Fatou's lemma we infere $u(\log u)_+ \in L^\infty(0,T; L^1(\R^3))$. However, it is easy to see that 
$ \xi(c_1 s)\leq C s( (\log s)_+ + 1) $ for $s\geq 0$, so $\xi(c_1 u) \in L^\infty(0,T; L^1(\R^3))$.
Therefore again by Jensen's inequality we deduce
$$ a_1(x,t)^3 = f_3(\xi(a_1(x,t))) \leq \int_{|x-y|<\rho}f_3\left(\frac{\sigma(x,t)}{c_1 |x-y|}\right)\frac{\xi(c_1 u(y,t))}{\sigma(x,t)}dy\quad x\in\R^3,~~t\in [0,T], $$
where $\sigma(x,t)\equiv \int_{|x-y|<\rho}\xi(c_1 u(y,t))dy$. Clearly $\sigma(x,t)\leq \sigma_1$, $|x|<R$, $0\leq t\leq T$, 
for some constant $\sigma_1>0$. However, $\sigma(x,t)$ is also uniformely positive. In fact, since $\xi(s)\geq C s$ for $s>0$, it holds
\begin{align*}
\sigma(x,t) &= \int_{|x-y|<\rho}\xi(c_1 u(y,t))dy\geq C\int_{|x-y|<\rho}u(y,t)dy = C\left( m -  \int_{|x-y|>\rho}u(y,t)dy\right)\\
&\geq C\left( m - \frac{1}{\rho^2}\int_{|x-y|>\rho}|x-y|^2 u(y,t)dy\right)\geq C\left( m - \frac{2}{\rho^2}\int_{\R^3}(|x|^2 + |y|^2 ) u(y,t)dy\right)\\
& = C\left( m - \frac{2(m |x|^2 + E(t))}{\rho^2}\right).
\end{align*}
The above inequality and the definition of $\rho$ imply that $\sigma(x,t)\geq \sigma_{0}>0$, $|x|<R$, $0\leq t\leq T$, for some positive constant $\sigma_{0}$.
The lower and upper bounds for $\sigma(x,t)$, as well as the fact that $f_3$ is nondecreasing, lead to
$$ a_1(x,t)^3 \leq C\int_{|x-y|<\rho}f_3\left(\frac{\sigma_1}{c_1 |x-y|}\right)\xi(c_1 u(y,t))dy\quad x\in\R^3,~~t\in [0,T]. $$
By integrating the above inequality in $B_R$ w.r.t. $x$ we obtain
\begin{align*}
\int_{|x|<R} a_1(x,t)^3 dx 
&\leq C\int_{|x|<R}\int_{|x-y|<\rho}f_3\left(\frac{\sigma_1}{c_1 |x-y|}\right)\xi(c_1 u(y,t))dy ~ dx\\
&=C\int_{|x|<R}\int_{|y|<\rho}f_3\left(\frac{\sigma_1}{c_1 |y|}\right)\xi(c_1 u(x-y,t))dy ~ dx\\
&=C\int_{|y|<\rho}f_3\left(\frac{\sigma_1}{c_1 |y|}\right)\int_{|x|<R}\xi(c_1 u(x-y,t))dx ~ dy\\
&\leq C\int_{|y|<\rho}f_3\left(\frac{\sigma_1}{c_1 |y|}\right)dy \int_{\R^3}\xi(c_1 u(x,t))dx,
\end{align*}
which implies, since $\xi(c_1 u) \in L^\infty(0,T; L^1(\R^3))$,
\begin{align}
\|a_1\|_{L^\infty(0,T; L^3(B_R))}^3 &\leq C \int_{|y|<\rho}f_3\left(\frac{\sigma_1}{c_1 |y|}\right)dy .\label{est.a1.new}
\end{align}
Let $\lambda = \sigma_1/c_1$. We want to show that the integral on the right-hand side of \eqref{est.a1.new} is convergent.
By using polar coordinates we get
$$ \int_{|y|<\rho}f_3\left(\frac{\sigma_1}{c_1 |y|}\right)dy = 4\pi\int_{0}^\rho f_3(\lambda r^{-1})r^2 dr . $$
By making the change of variables $\lambda r^{-1} = \xi(s) = (1+s)\log(1+s)$ and recalling the definition of $f_3$ we deduce
\begin{align*}
\int_{|y|<\rho}f_3\left(\frac{\sigma_1}{c_1 |y|}\right)dy 
&= \frac{4\pi}{\lambda^3}\int_{s_0}^{\infty}\frac{s^3}{(1+s)^4}\frac{1+\log(1+s)}{(\log(1+s))^4}ds\\
&\leq \frac{4\pi}{\lambda^3}\int_{s_0}^{\infty}\frac{(\log(1+s))^{-3} + (\log(1+s))^{-4} }{1+s}ds \\
& = \frac{4\pi}{\lambda^3}\left[ -\frac{1}{2}(\log(1+s))^{-2} -\frac{1}{3}(\log(1+s))^{-3} \right]_{s=s_0}^{s\to\infty} < \infty .
\end{align*}
So $a_1\in L^\infty(0,T; L^3(B_R))$. Since we already knew that $a_2\in L^\infty(0,T; L^3(B_R))$ and ${4\pi} a = a_1 + a_2$,
this implies, given also the arbitrariety of $R>0$, that $a\in L^\infty(0,T; L^3_{loc}(\R^3))$.

The proof that $\na a \in L^\infty(0,T; L^{3/2}_{loc}(\R^3))$ follows the same argument, the only difference being that the function
$f_{3/2}$ is to be employed in place of $f_3$. Therefore \eqref{a.reg.new} has been proved.

As a consequence of \eqref{a.reg.new}, we can integrate \eqref{landau.weak} by parts and deduce
by a density argument that the weak formulation \eqref{landau.w} holds.

Finally, $u\in W^{1,r}(0,T; W^{2,\frac{r}{r-1}}(\R^{3},\gamma^{-\frac{1}{3(r-1)}}(x)dx)')
\hookrightarrow C([0,T],W^{2,\frac{r}{r-1}}(\R^{3},\gamma^{-\frac{1}{3(r-1)}}(x)dx)')$, so the limit $\lim_{t\to 0}u(t)=u_0$ 
in $W^{2,\frac{r}{r-1}}(\R^{3},\gamma^{-\frac{1}{3(r-1)}}(x)dx)'$ follows.
This finishes the proof of Thr.~\ref{thr.ex}.

\begin{rem}
We point out that, while almost all the computations in the proof of Thr.~\ref{thr.ex} can be adapted to the case $d\geq 3$,
ineq.~\eqref{critical} becomes 
\begin{align*}
&\left( \int_\Bt a_\tau[\utau]^p \utau dx \right)^{1/p} \\ 
&\leq C_p \left(1 - (1+\Etau(t))^{d/2-1}D_\tau H^\tau[\utau(t)] \right)^{1/p}\left( \int_\Bt (1+|x|)^{d-2} u(x,t) dx \right)^{1/p} .
\end{align*}
If $d>4$ we cannot control the right-hand side of the above inequality by means of a power of $\Etau$.
This is the only reason why we have assumed $d=3$.
\end{rem}

\begin{proof}[Proof of Corollary~\ref{coro.ex}]
We repeat the calculations done in the proof of Section ``{\em Weak formulation}'' of {\em Step 3}. Since $|A[f]|\leq a[f]$
then $A[f]\in L^{\infty}(0,T; L^{3}_{loc}(\R^{3}))$. The proof of $\na a[f]\in L^{\infty}(0,T; L^{3/2}_{loc}(\R^{3}))$ is exactly the same.
Starting from the weak formulation in \cite[Corollary 1.1]{Des2015} we can integrate by parts and obtain that \eqref{landau.true.w}
holds true for $\phi$ smooth enough. A standard density argument shows
that the test functions can be chosen in the space $L^{\infty}(0,T; W^{1,\infty}_{c}(\R^{3}))$. This finishes the proof.
\end{proof}

%
%

\section{Proof of Theorem \ref{thr.reg}}\label{section_theorem2}

We will make use of the following two-weight Sobolev inequality:

\begin{lemma}[Weighted Sobolev inequality]\label{lem:Inequalities Sobolev weight aII}
Let $$a[u^{(\tau)}](x) = \int_{\Bt}\frac{u^{(\tau)}(y)}{{4\pi} |x-y|}dy,$$ with $u^{(\tau)}$ a solution to \eqref{landau.tau}. There exists an universal constant $C$ such that any smooth function $\phi$ satisfies 
    \begin{align*}
     \left ( \int_{I}\int \phi^q a[u^{(\tau)}]\;dxdt \right )^{2/q}& \leq C\left ( \int_{I} \int a[u^{(\tau)}] |\nabla \phi|^2 \;dxdt +  \sup \limits_{I} \int  \phi^{2}\;dx \right ),
    \end{align*}
    with
    \begin{align*}
    q \in \left( 1 , 2\left(1+\frac{2}{3}\right)\right). 
    \end{align*}
  \end{lemma}
  \begin{proof}
  The proof can be found in \cite[Section 4]{GG17} but we sketch it here for completeness.  
  
  We first recall a weighted inequality proven in \cite[Theorem 1.5]{CW2} and \cite[Theorem 1]{SW} , which states that any smooth function $\phi$ compactly supported in $Q \subset \R^3$ satisfies 
  \begin{align}\label{SWsobolev_weighted_ineq}
  \left( \int_{Q}\phi^qw_1\;dv\right)^{1/q}\le C\left( \int_{Q} |\nabla \phi|^pw_2\;dv\right)^{1/p},
  \end{align} 
   provided $1<p\le  q<+\infty$ and $w_1(v)$, $w_2(v)$ are $\mathcal{A}_1$-weights such that 
  \begin{align}\label{cond_ineq}
    \left ( \frac{|Q'|}{|Q|}\right )^{1/3}\left(\frac{\int_{Q'} w_1\;dv }{ \int_{Q} w_1\;dv  }\right)^{1/q}\le C\left(\frac{\int_{Q'} w_2\;dv }{ \int_{Q} w_2\;dv  }\right)^{1/p}   ,
  \end{align}
  for all cubes $Q'\subset  8Q$. A function $\omega(x)$ is an $\mathcal{A}_1$-weight if there exists a constant $C$ such that 
  $$
  \frac{1}{|Q|} \int \omega \;dx \le C \omega (x_0), \quad \textrm{for almost every $x_0$ in $Q$}.
  $$
   Functions of the form $\frac{1}{|x|^m}$ with $0\le m <3$ are $\mathcal{A}_1$-weights in $\R^3$. 

For our inequality we chose $w_1= a^m[u^{(\tau)}]$ and $w_2 = a[u^{(\tau)}]$ for some $m<3$. Since $\frac{1}{|w|}$ is $\mathcal{A}_1$ we have 
\begin{align*}
\frac{1}{|Q_r|} \int_{Q_r}  a[u^{(\tau)}] \;dx = \frac{c}{|Q_r|} \int_{\Bt} u^{(\tau)} \int_{Q_r}  \frac{1}{|w-y|}\;dwdy \le c  \int_{\Bt} \frac{u^{(\tau)}(y)}{|x-y|}\;dy  = c a[u^{(\tau)}] (x)
\end{align*}
for almost every $x \in Q_r$. We deduce that $a[u^{(\tau)}]$ is an $\mathcal{A}_1$-weight.  Similarly, Minkowski's inequality implies 
\begin{align*}
\left( \frac{1}{|Q_r|} \int_{Q_r}  a^m[u^{(\tau)}] \;dw\right)^{\frac{1}{m}} &\le c  \int_{\Bt} u^{(\tau)} \left( \frac{1}{|Q_r|} \int_{Q_r} \frac{1}{|w-y|^m}\;dw \right)^{\frac{1}{m}}dy\\
  &\le \int_{\Bt}\frac{ u^{(\tau)}(y)}{|x-y|}\;dy = C  a[u^{(\tau)}](x),
\end{align*}
for almost every $x \in Q_r$, with $C$ universal constant. Therefore $a^m[u^{(\tau)}]$ is also an $\mathcal{A}_1$-weight.  Taking the average over $Q_r$ on both sides of the last inequality we get 
\begin{align}\label{rev_Holder}
\left( \frac{1}{|Q_r|} \int_{Q_r}  a^m[u^{(\tau)}] \;dw\right)^{\frac{1}{m}} \le C \frac{1}{|Q_r|} \int_{Q_r}  a[u^{(\tau)}](x)\;dx.
\end{align}
 We now apply respectively H\"older's inequality and (\ref{rev_Holder}) in $Q'$ to get 
   \begin{align*}
    \int_{Q}a_{\tau}[u^{(\tau)}] \;dv & \leq |Q|^{1-\frac{1}{m}} \left ( \int_{Q}a^m[u^{(\tau)}]\;dv\right )^{\frac{1}{m}},\\	  
    \left ( \int_{Q'} a^m[u^{(\tau)}]\;dv\right )^{\frac{1}{m}} & \leq C|Q'|^{\frac{1}{m}-1}\int_{Q'} a[u^{(\tau)}]\;dv,
  \end{align*}
  with $C$ universal constant. Dividing the left hand side of the second inequality by the right hand side of the first inequality, we arrive at 
  \begin{align*}
    \left (\frac{\int_{Q'} a^m[u^{(\tau)}]\;dv}{\int_{Q}a^m[u^{(\tau)}]\;dv} \right )^{\frac{1}{m}} \leq C\left ( \frac{|Q'|}{|Q|}\right )^{\frac{1}{m}-1}\frac{\int_{Q'}a[u^{(\tau)}]\;dv}{\int_{Q}a[u^{(\tau)}]\;dv}.
  \end{align*}
  Taking the square root on both sides, and multiplying by $(|Q'|/|Q|)^{\frac{1}{3}}$ leads to
  \begin{align*}
    \nonumber\left ( \frac{|Q'|}{|Q|}\right )^{\frac{1}{3}}\left (\frac{\int_{Q'} a^m[u^{(\tau)}]\;dv}{\int_{Q}a^m[u^{(\tau)}] \;dv} \right )^{\frac{1}{2m}}&  \leq C\left ( \frac{|Q'|}{|Q|}\right )^{\frac{1}{3}+\frac{1-m}{2m}}\left ( \frac{\int_{Q'}a[u^{(\tau)}]\;dv}{\int_{Q}a[u^{(\tau)}]\;dv}\right )^{\frac{1}{2}}\\
   & \leq C\left ( \frac{\int_{Q'}a[u^{(\tau)}]\;dv}{\int_{Q}a[u^{(\tau)}]\;dv}\right )^{\frac{1}{2}} \le C
  \end{align*}
  since 
  $$\frac{1}{3}+\frac{1-m}{2m} > 0, \quad \textrm{for} \quad m < 3.$$
  Then, we may apply \eqref{SWsobolev_weighted_ineq} in the case where $p=2$, $q=2m$ with $m<3$  and get 
   \begin{align*}
  \left( \int_{Q}\phi^{2m} a^m[u^{(\tau)}] \;dv\right)^{1/2m}\le C\left( \int_{Q} a[u^{(\tau)}] |\nabla \phi|^2 \;dv\right)^{1/2}.
  \end{align*}   
  Since the constant $C$ does not depend on $Q$, density's argument allows to write 
    \begin{align*}
  \left( \int \phi^{2m} a^m[u^{(\tau)}] \;dv\right)^{1/2m}\le C\left( \int  a[u^{(\tau)}] |\nabla \phi|^2 \;dv\right)^{1/2},
  \end{align*}   
  for any function $\phi$ defined in $\R^3$. 
  For $q = 4-\frac{2}{m}$ interpolation yields
  \begin{align*}
    \int \phi^q a[u^{(\tau)}]\;dv \leq \left (  \int  \phi^{2m} a^m[u^{(\tau)}] \;dv\right )^{\frac{1}{m}} \left (\int \phi^2\;dv \right )^{1-\frac{1}{m}}.
  \end{align*}
   We now integrate in the time interval $I$ and use the above estimate to get 
  \begin{align*}
    \int_I\int \phi^q a[u^{(\tau)}]\;dvdt & \leq \sup \limits_{I} \left(\int \phi^{2}\;dv \right )^{1-\frac{1}{m}} \int_{I}\left (  \int  \phi^{2m} a^m[u^{(\tau)}] \;dv \right )^{\frac{1}{m}}\;dt \\
     & \leq C \sup \limits_{I} \left(\int \phi^{2}\;dv \right )^{1-\frac{1}{m}}  \int_{I}\int a[u^{(\tau)}] |\nabla \phi|^2 \;dvdt, 
    \end{align*}
  which implies
    \begin{align*}
    \int_{I}  \int\phi^q a[u^{(\tau)}]\;dvdt & \leq C\left ( \int_{I} \int a[u^{(\tau)}] |\nabla \phi|^2 \;dvdt + \sup \limits_{I} \int \phi^{2}\;dv \right )^{\frac{q}{2}}.
    \end{align*}  
    Since $m<3$ we have that $q<2+ \frac{4}{3}$.

  \end{proof}

Theorem ~\ref{thr.reg} will be shown for piecewise constant in time solutions $\utau$ to the discretized problem \eqref{landau.tau};
then the a.e. convergence of $\utau$, $a[\utau]$ will yield the statement for solutions to \eqref{landau} satisfying the properties stated in Theorem ~\ref{thr.ex}.
 
We assume that the estimates for $\utau$ derived in the previous section are satisfied. In particular,
we assume that $c_1\leq\int_\Bt \utau(x,t) dx\leq c_2$ and $\int_\Bt |x|^2 \utau(x,t) dx\leq c_3$,
$t\in [0,T]$, for suitable positive constants $c_1, c_2, c_3$ which do not depend on $\tau$;
that $\sqrt{\utau}$ is uniformely bounded in $L^2(0,T; H^1(\Bt,\gamma(x)dx))$; and that \eqref{June20.bounds.u} holds. 
We also assume that $\utau\in L^\infty(0,T; W^{1,4}(\Bt))$, although this bound is not uniform in $\tau$.
Finally, for the sake of simplicity we write $u\equiv\utau$ and $a\equiv a[\utau]$ .

We also recall that $D_\tau$ is the backward discrete time derivative defined in \eqref{Dtau.def}.

   \begin{prop}\label{prop:energy_identity} The following inequality holds:
    \begin{align}\label{est3-1}
      & D_\tau\int \eta^2 u^p\;dx + \frac{4(p-1)}{p}\int a|\nabla (\eta u^{p/2})|^2\;dx {+\frac{p(p-1)\tau}{2}\int \frac{u^{p-2}}{u^3}|\na u|^4 \eta^2 dx}\\
      &\qquad\leq \textnormal{(I)} + \textnormal{(II)} 
      {+ C\tau\int \eta^2 u^p\; dx + C(p)\tau\int \left(1+\frac{|\na\eta|^{4p}}{\eta^{4p}}\right)\eta^2 dx},\nonumber
    \end{align}
    where 
    \begin{align*}		  
      \textnormal{(I)} & := \frac{4(p-2)}{p}\int u^{p/2}(a\nabla (\eta u^{p/2}),\nabla \eta)\;dx+\frac{4}{p}\int u^{p}(a\nabla \eta,\nabla \eta)\;dx,\\
  	\textnormal{(II)} & := \int u^p (\nabla a,\nabla (\eta^2))\;dx+(p-1)\int u\eta^2 u^p\;dx.
    \end{align*}   
  \end{prop}
  \begin{proof}
    Consider 
    $$
\psi= p  \; \eta^2\; u^{p-1}
 $$
 as test function for (\ref{landau}). Since $s\mapsto s^p$ is convex, a direct computation yields, 
    \begin{align*}
D_\tau\int \eta^2 u^p\;dx & \leq p \int \eta^2 u^{p-1} D_\tau u\;dx \\
        & = - p\int (a\nabla u,\nabla (\eta^2 u^{p-1}))\;dx
        +p \int (u\nabla a,\nabla (\eta^2 u^{p-1}))\;dx\\
	&~{-p\tau\int ( (\log u)^3\eta^2 u^{p-1} + |\na\log u|^2\na\log u\cdot\na(\eta^2 u^{p-1}) )dx}\\
         &= \widetilde{\textnormal{(I)}} + \textnormal{(II)} {+\textnormal{(III)}}.
    \end{align*}
    Expanding the first integral, we have the expression:
    \begin{align*}
    \int (a\nabla u,\nabla (\eta^2 u^{p-1}))\;dx =   \int (p-1) \eta^2 u^{p-2}(a\nabla u,\nabla u)+2 u^{p-1}\eta(a\nabla u,\nabla \eta)\;dx.
    \end{align*}	  
    Let us rewrite this expression in a more convenient form. Note the elementary identity
    \begin{align*}
      (a\nabla (\eta u^{p/2}),\nabla (\eta u^{p/2}))  
      & = \frac{p^2}{4} u^{p-2}\eta^2 (a\nabla u,\nabla u)+p\eta u^{p-1}(a\nabla u,\nabla \eta)+u^{p}(a\nabla \eta,\nabla \eta),
    \end{align*}
    and use it to write,
    \begin{align*}		
      & (p-1)\eta^2 u^{p-2}(a\nabla u,\nabla u)+2u^{p-1}\eta (a\nabla u,\nabla \eta) \\
      & = \frac{4(p-1)}{p^2}(a\nabla (\eta u^{p/2}),\nabla (\eta u^{p/2}))\\
      & \;\;\;\; - \frac{(2p-4)}{p} u^{p-1}\eta(a\nabla u,\nabla \eta)-\frac{4(p-1)}{p^2} u^p(a\nabla \eta,\nabla \eta).	  	
    \end{align*}
    Further, another elementary identity says
    \begin{align*}	
      u^{p-1}\eta (a\nabla u,\nabla \eta)  
	  = \frac{2}{p} u^{p/2} (a  \nabla (\eta u^{p/2}),\nabla \eta )-\frac{2}{p}u^p (a\nabla \eta,\nabla \eta). 		
    \end{align*}		
    Combining the above, it follows that
    \begin{align*}	
      & (p-1)\eta^2 u^{p-2}(a\nabla u,\nabla u)+2u^{p-1}\eta (a\nabla u,\nabla \eta)\\
      & = \frac{4(p-1)}{p^2}(a\nabla (\eta u^{p/2}),\nabla (\eta u^{p/2}))\\	
      & \;\;\;\;-\frac{4(p-2)}{p^2}u^{p/2}(a\nabla (\eta u^{p/2}),\nabla \eta)-\frac{4}{p^2}u^p(a\nabla \eta,\nabla \eta).
    \end{align*}	
    In particular,
    \begin{align*}	
   \widetilde{\textnormal{(I)}}=    & -\frac{4(p-1)}{p}\int (a\nabla (\eta u^{p/2}),\nabla (\eta u^{p/2}))\;dx\\
      & +\frac{4(p-2)}{p}\int u^{p/2}(a\nabla (\eta u^{p/2}),\nabla \eta)\;dx+\frac{4}{p}\int u^{p}(a\nabla \eta,\nabla \eta)\;dx.
    \end{align*}	
    Thus,
    \begin{align*}	
      & \frac{d}{dt}\int \eta^2 u^p\;dx + \frac{4(p-1)}{p}\int (a\nabla (\eta u^{p/2}),\nabla (\eta u^{p/2}))\;dx \\
      & = \frac{4(p-2)}{p}\int u^{p/2}(a\nabla (\eta u^{p/2}),\nabla \eta)\;dx+\frac{4}{p}\int u^{p}(a\nabla \eta,\nabla \eta)\;dx+ p\int (u\nabla a,\nabla (\eta^2u^{p-1} ) )\;dx.
    \end{align*}	
    We now analyze ${\textnormal{(II)}}$. Since
    \begin{align*}
      (\nabla a,u\nabla (\eta^2 u^{p-1})) & =  uu^{p-1}(\nabla a,\nabla (\eta^2))+ (p-1)uu^{p-2}\eta^2(\nabla a,\nabla u)\\
      & = uu^{p-1}(\nabla a,\nabla (\eta^2))+(p-1)( u^{p-1})\eta^2 (\nabla a,\nabla u)\\
      & = (u^p)(\nabla a,\nabla (\eta^2))+\eta^2 (\nabla a,\nabla ( \frac{p-1}{p}u^p) ),
    \end{align*}
    it follows that
    \begin{align*}
      \textnormal{(II)}  & = p\int (u^p )(\nabla a,\nabla (\eta^2))\;dx -p\int \left ( \frac{p-1}{p}u^{p}\right )\dive(\eta^2\nabla a)\;dx.
    \end{align*}		
    {From the above inequality and the Poisson equation it follows}
    \begin{align*}			  
  	\textnormal{(II)} & = p\int (u^p) (\nabla a,\nabla (\eta^2))\;dx-\int ((p-1)u^p) (\nabla a,\nabla (\eta^2))\;dx\\
        & \;\;\;\;+\int u \eta^2 ((p-1)u^p) \;dx\\
        & = \int u^p (\nabla a,\nabla (\eta^2))\;dx+\int u\eta^2 \left ( (p-1)u^p \right )\;dx.
    \end{align*}	  
    
    
Let us now consider the third integral. It holds
\begin{align*}
\mbox{(III)} &= -p\tau\int ( (\log u)^3\eta^2 u^{p-1} + |\na\log u|^2\na\log u\cdot\na(\eta^2 u^{p-1}) )dx\\
&=  -p(p-1)\tau\int \frac{u^{p-2}}{u^3}|\na u|^4 \eta^2 dx 
- p\tau\int u^{p-1}|\na\log u|^2\na\log u\cdot \na(\eta^2) dx\\
&\qquad -p\tau\int ( (\log u)^3\eta^2 u^{p-1}dx .
\end{align*}
The second integral of $\mbox{(III)}$ can be estimated as 
\begin{align*}
& - p\tau\int u^{p-1}|\na\log u|^2\na\log u\cdot \na(\eta^2) dx\\
&\leq p\tau\int \frac{u^{p-1}}{u^3}\frac{|\na u|^3}{u^{3/4}}\eta^{3/2}\cdot u^{3/4}|\na(\eta^2)|\eta^{-3/2} dx\\
& \leq \frac{p(p-1)\tau}{2}\int \frac{u^{p-2}}{u^3}|\na u|^4 \eta^2 dx + 
\frac{p\tau}{2(p-1)}\int \frac{u^{p+2}}{u^3} \frac{|\na(\eta^2)|^4}{\eta^6} dx\\
& \leq \frac{p(p-1)\tau}{2}\int \frac{u^{p-2}}{u^3}|\na u|^4 \eta^2 dx + 
\frac{C p\tau}{p-1}\int u^{p-1} \frac{|\na\eta|^4}{\eta^2} dx .
\end{align*}
By applying Young inequality to bound the second integral we get
\begin{align*}
\int u^{p-1} \frac{|\na\eta|^4}{\eta^2} dx 
&= \int u^{p-1}\eta^{2(p-1)/p} \cdot \frac{|\na\eta|^4}{\eta^2}\eta^{-2(p-1)/p} dx\\
&\leq \frac{p-1}{p}\int u^p\eta^2 dx + \frac{1}{p}\int |\na\eta|^{4p}\eta^{2-4p}dx .
\end{align*}
Moreover, 
\begin{align*}
-p\tau\int  (\log u)^3\eta^2 u^{p-1}dx &\leq -p\tau\int_{\{u\leq 1\}} (\log u)^3\eta^2 u^{p-1}dx
\leq C p \tau\int\eta^2 dx ,
\end{align*}
where $C \equiv \sup_{0<s\leq 1}s^{p-1} (\log(1/s))^3 < \infty$ with $p>1$. 
Therefore we conclude
\begin{align*}
\mbox{(III)} &\leq -\frac{p(p-1)\tau}{2}\int \frac{u^{p-2}}{u^3}|\na u|^4 \eta^2 dx + {{ C\tau\int u^p\eta^2 dx 
+C(p)\tau\int \left(1+\frac{|\na\eta|^{4p}}{\eta^{4p}}\right)\eta^2 dx}} .
\end{align*}

This finishes the proof of the Lemma.
  \end{proof}	

\begin{lemma} \label{lemma:2}
Let $p>1$, then we have the inequality
    \begin{align*}
      & \frac{d}{dt}\int \eta^2 u^p\;dx + \frac{(p-1)}{p}\int a|\nabla (\eta u^{p/2})|^2\;dx\\
      & \leq {{(p-1)\int \eta^2 u^{p+1}\;dx}} \\
      & \;\;\;\;+C(p) \int u^p (a\nabla \eta,\nabla \eta) \;dx-\int u^p \eta \textrm{Tr}\;(aD^2\eta))\;dx\\
&   \;\;\;\;    +{{ C\tau\int u^p\eta^2 dx 
 +C(p)\tau\int \left(1+\frac{|\na\eta|^{4p}}{\eta^{4p}}\right)\eta^2 dx}} ,
         \end{align*}
    where $C(p)$ denotes a constant that is bounded when $p>1$.
  \end{lemma}

  \begin{proof}	
    We proceed to bound from above the first term $\textnormal{(I)}$ and the first term of $\textnormal{(II)}$ resulting from Proposition \ref{prop:energy_identity}.  The aim is to estimate these terms as 
    \begin{align*}
     \frac{4(p-2)}{p}\int u^{p/2}(a\nabla (\eta u^{p/2},\nabla \eta)\;dx& +\int u^p (\nabla a,\nabla (\eta^2))\;dx  \\
     \le & c_1 \int (a\nabla (\eta u^{p/2}),\nabla (\eta u^{p/2}))\;dx + \textrm{\em lower order terms,}
\end{align*}
    where $c_1 < \frac{4(p-1)}{p}$.
   For the first term we use Cauchy-Schwarz inequality 
    \begin{align}
      & \left | \frac{4(p-2)}{p}(a\nabla (\eta u^{p/2}),u^{p/2}\nabla \eta) \right | \nonumber \\
      & \leq \frac{2(p-1)}{p}(a\nabla (\eta u^{p/2}),\nabla (\eta u^{p/2}))+ \frac{2(p-2)^2}{p(p-1)}u^p(a\nabla \eta,\nabla \eta). \label{est2-1}
    \end{align}			
    For the first term in $\textnormal{(II)}$ we use the identity
    \begin{align*}
      \dive(a u^p \nabla (\eta^2)) = a \dive(u^p\nabla (\eta^2))+u^p(\nabla a,\nabla (\eta^2)),
    \end{align*} 
    and conclude that
    \begin{align*}
      \int u^p (\nabla a,\nabla (\eta^{2}))\;dx & = -\int a \dive(u^p\nabla (\eta^2))\;dx\\
       & = -\int a u^p \Delta (\eta^2)\;dx- \int (a\nabla u^p,\nabla \eta^2)\;dx.
    \end{align*}
    Since
    \begin{align*}
       \eta \nabla u^{p/2} = \nabla (\eta u^{p/2})- u^{p/2}\nabla \eta,	
    \end{align*}
    {Young's inequality} yields
    \begin{align*}
       &{-}\int (a \nabla u^p,\nabla \eta^2)\;dx ~ {= -4\int u^{p/2}(a\eta\na u^{p/2}, \nabla\eta)}\\
       &\qquad = {-}4\int u^{p/2}(a \nabla (\eta u^{p/2}),\nabla \eta)\; dx +4\int u^p (a\nabla \eta,\nabla \eta)\;dx \\
       &\qquad \leq 2\varepsilon\int (a\nabla (\eta u^{p/2}),\nabla (\eta u^{p/2}) )\;dx + \left(\frac{2}{\varepsilon}+4\right)\int u^{p}(a\nabla \eta,\nabla\eta)\;dx .
    \end{align*}
    Thus
    \begin{align}\label{est1-1}
      \int u^p (\nabla a,\nabla (\eta^{2}))\;dx  \le &-\int u^p \textrm{Tr} (a D^2(\eta^2))\;dx +  2\varepsilon\int (a\nabla (\eta u^{p/2}),\nabla (\eta u^{p/2}) )\;dx \nonumber\\
    &  + \left(\frac{2}{\varepsilon}+4\right)\int u^{p}(a\nabla \eta,\nabla\eta)\;dx.
      \end{align}
  Substituting (\ref{est1-1}) and  (\ref{est2-1}) into  (\ref{est3-1}) we get 
   \begin{align*}
       \frac{d}{dt}\int \eta^2 u^p\;dx &+ \frac{(p-1)}{p} \int (a\nabla (\eta u^{p/2}),\nabla (\eta u^{p/2}))\;dx \\
      \le &C(p)  \int u^{p}(a\nabla \eta,\nabla \eta)\;dx+(p-1)\int \eta^2 u u^p\;dx -\int u^p \textrm{Tr} (a D^2(\eta^2))\;dx,
    \end{align*}   
by choosing $\varepsilon < \frac{p-1}{2p}$.	
  \end{proof}

\begin{lemma}\label{lemma3}
 We have
\begin{align*}
(p-1)\int_t^T \int \eta^2 u^{p+1}\;dxds  \le & \; \varepsilon (p-1)\int_t^T \int  a |\nabla(\eta u^{p/2})|^2 \;dxds + C(\varepsilon,p)\int_t^T\int	\eta^2 u^{p}\;dxds.
\end{align*}
\end{lemma}

\begin{proof}
We use here the $\varepsilon$-Poincare's inequality (\ref{eqn:epsilon_Poincare_inequality strongerII}) with 
$$
\phi = \eta u^{p/2}
$$
and get 
 \begin{align*}
  \begin{array}{l}
	 \int \eta^2 u^{p+1}\;dx \leq \varepsilon \int a |\nabla(\eta u^{p/2})|^2 \;dx + C(\varepsilon)\int 	\eta^2 u^{p}\;dx.
  \end{array}	  
\end{align*}


\end{proof}

\begin{cor}  \label{cor_iterat}
  Fix times $0<T_1<T_2<T_3<T$, $p>1$ and a cut-off function $\eta(v)$. Then, we have the following inequality
  \begin{align*}
   \sup \limits_{T_2 \leq t\leq T_3} \left \{ \int (\eta u^{p/2})^2\;dx\right\} &+ \frac{(p-1)}{4p}  \int_{T_2}^{T_3}\int a|\nabla (\eta u^{p/2})|^2\;dxdt\\
      \le  &\; \left(\frac{1}{T_2-T_1} + C(p,\varepsilon)\right)\int_{T_1}^{T_3} \int \eta^2u^p\;dxdt \\
            &+C(p) \int_{T_1}^{T_3}\int  u^p (a\nabla \eta,\nabla \eta) \;dxdt\\
     & + \int_{T_1}^{T_3}\int a u^p \eta |\Delta \eta|\;dxdt\\
     & + {{ C\tau\int u^p\eta^2 dx 
+C(p)\tau\int \left(1+\frac{|\na\eta|^{4p}}{\eta^{4p}}\right)\eta^2 dx}} 	.
    \end{align*}	

    \end{cor}  
  
\begin{proof}  
  We start with the bound found in Lemma \ref{lemma:2} 
  \begin{align*}
      & \frac{d}{dt}\int \eta^2 u^p\;dx + \frac{(p-1)}{p}\int a|\nabla (\eta u^{p/2})|^2\;dx\\
      & \leq {{(p-1)\int \eta^2 u^{p+1}\;dx}} +C(p) \int u^p (a\nabla \eta,\nabla \eta) \;dx-\int a u^p \eta \Delta\eta\;dx .
         \end{align*}
         
         

    Integrating this inequality from $t_1$ to $t_2$ shows that the term  
    \begin{align*}
      \int \eta^2 u^p(t_2)\;dx-\int \eta^2u^p(t_1)\;dx + \frac{(p-1)}{p}\int_{t_1}^{t_2}\int a|\nabla (\eta u^{p/2})|^2\;dxdt
    \end{align*}	
    is no larger than	
    \begin{align*}	
     (p-1) \int_{t_1}^{t_2}\int \eta^2 u^{p+1}\;dxdt
      +C(p) \int_{t_1}^{t_2}\int  u^p (a\nabla \eta,\nabla \eta) \;dxdt- \int_{t_1}^{t_2}\int a u^p \eta \Delta \eta\;dxdt. 	
    \end{align*}	
    
For a fixed $t_2 \in (T_2,T_3)$, we take the average with respect to $t_1\in (T_1,T_2)$ in both sides of the inequality. This yields 
  \begin{align*}
  \frac{1}{T_2-T_1}\int_{T_1}^{T_2} \int \eta^2 u^p(t_2)\;dxdt_1 &+ \frac{(p-1)}{p}\frac{1}{T_2-T_1}\int_{T_1}^{T_2}  \int_{t_1}^{t_2}\int a|\nabla (\eta u^{p/2})|^2\;dxdtdt_1\\
      \le  &\; \frac{1}{T_2-T_1}\int_{T_1}^{T_2} \int \eta^2u^p(t_1)\;dxdt_1 \\
      &+ (p-1) \frac{1}{T_2-T_1}\int_{T_1}^{T_2} \int_{t_1}^{t_2}\int \eta^2 u^{p+1}\;dxdtdt_1\\
      &+C(p) \frac{1}{T_2-T_1}\int_{T_1}^{T_2}\int_{t_1}^{t_2}\int  u^p (a\nabla \eta,\nabla \eta) \;dxdtdt_1\\
     & - \frac{1}{T_2-T_1}\int_{T_1}^{T_2} \int_{t_1}^{t_2}\int a u^p \eta \Delta \eta\;dxdtdt_1,
    \end{align*}	 
which implies 
\begin{align*}
  \int \eta^2 u^p(t_2)\;dx &+ \frac{(p-1)}{p}  \int_{T_2}^{t_2}\int a|\nabla (\eta u^{p/2})|^2\;dxdt\\
      \le  &\; \frac{1}{T_2-T_1}\int_{T_1}^{T_2} \int \eta^2u^p(t_1)\;dxdt_1 + (p-1) \int_{T_1}^{t_2}\int \eta^2 u^{p+1}\;dxdt\\
      &+C(p) \int_{T_1}^{t_2}\int  u^p (a\nabla \eta,\nabla \eta) \;dxdt + \int_{T_1}^{t_2}\int a u^p \eta |\Delta \eta|\;dxdt. 	
    \end{align*}	
   Since this holds for every $t_2\in (T_2,T_3)$, this implies the inequality
  \begin{align*}
   \sup \limits_{T_2 \leq t\leq T_3} \left \{ \int \eta^2 u^p(t)\;dx\right\} &+ \frac{(p-1)}{p}  \int_{T_2}^{T_3}\int a|\nabla (\eta u^{p/2})|^2\;dxdt\\
      \le  &\; \frac{1}{T_2-T_1}\int_{T_1}^{T_3} \int \eta^2u^p(t)\;dxdt + (p-1) \int_{T_1}^{T_3}\int \eta^2 u^{p+1}\;dxdt\\
            &+C(p) \int_{T_1}^{T_3}\int  u^p (a\nabla \eta,\nabla \eta) \;dxdt + \int_{T_1}^{T_3}\int a u^p \eta |\Delta \eta|\;dxdt. 	
    \end{align*}	
  As the last step we use Lemma \ref{lemma3} with $\varepsilon < \frac{p-1}{4p^2}$ and get 
   \begin{align*}
   \sup \limits_{T_2 \leq t\leq T_3} \left \{ \int \eta^2 u^p(t)\;dx\right\} &+ \frac{(p-1)}{4p}  \int_{T_2}^{T_3}\int a|\nabla (\eta u^{p/2})|^2\;dxdt\\
      \le  &\; \frac{1}{T_2-T_1}\int_{T_1}^{T_3} \int \eta^2u^p(t)\;dxdt + C(p,\varepsilon) \int_{T_1}^{T_3}\int \eta^2 u^p\;dxdt\\
            &+C(p) \int_{T_1}^{T_3}\int  u^p (a\nabla \eta,\nabla \eta) \;dxdt  + \int_{T_1}^{T_3}\int a u^p \eta |\Delta \eta|\;dxdt. 	
    \end{align*}

 \end{proof}
 
 Let us fix now $R$ and $r$ with $1<r<R$ and denote respectively with $\eta_r$ and with $\eta_R$ the cut-off function compactly supported in $B_r$ and $B_R$. Let moreover $ \eta_R =1$ on $B_r$, and $|\nabla \eta_R | \le c \frac{1}{(R-r)} \eta_R$ and $\Delta \eta_R \le c \frac{1}{(R-r)^2}  \eta_R$.

 The last corollary says that for $T_3=T$
   \begin{align*}
   \sup \limits_{T_1 \leq t\leq T} \left \{ \int (\eta_r u^{p/2})^2\;dx\right\} &+ \frac{(p-1)}{4p}  \int_{T_1}^{T}\int a[u]|\nabla (\eta_r u^{p/2})|^2\;dxdt\\
      \le  &\; \left(\frac{1}{T_1-T_0} + C(p,\varepsilon)\right)\int_{T_0}^{T} \int \eta_r^2u^p\;dxdt \\
            &+C(p) \int_{T_0}^{T}\int  u^p a[u]| \nabla \eta_r|^2  \;dxdt\\
     & + \int_{T_0}^{T}\int a[u] u^p \eta_r |\Delta \eta_r|\;dxdt + {{C\tau}}. 	
    \end{align*}	
 
 Using Lemma \ref{lem:Inequalities Sobolev weight aII} for any 
 \begin{align*}
    q \in \left( 1 , 2\left(1+\frac{2}{3}\right)\right),
    \end{align*}
    we have
     \begin{align*}
     \left ( \int_{T_1}^T \int a[u] (\eta_r u^{p/2})^q \;dxdt \right )^{2/q}\le  &\; \left(\frac{1}{T_1-T_0} + C(p,\varepsilon)\right)\int_{T_0}^{T} \int \eta_r^2u^p\;dxdt \\
            &+C(p) \int_{T_0}^{T}\int  a[u] u^p | \nabla \eta_r|^2  \;dxdt\\
     & + \int_{T_0}^{T}\int a[u] u^p \eta_r |\Delta \eta_r|\;dxdt+ {{C\tau}} \\
     \le  &\; C(R) \left(\frac{1}{T_1-T_0} + 1\right)\int_{T_0}^{T} \int a[u]\eta_r^2u^p\;dxdt \\
            &+C(p) \frac{1}{r^2}\int_{T_0}^{T} \int  a[u] u^p \eta_r^2  \;dxdt + {{C\tau}},   
      \end{align*}	
 taking into account the bound from below of $a[u]$ (\ref{a_tau_bound_below}). Since $\eta_R =1$ on the support of $\eta_r$ we can write $ \eta_r^2  \le \eta_R^q$ and get  
 \begin{align*}
     \left ( \int_{T_1}^T \int a[u] (\eta_r u^{p/2})^q \;dxdt \right )^{2/q} \le & C(R,p) \left(\frac{1}{T_1-T_0} + 1\right)\int_{T_0}^{T} \int a[u]\eta_R^q u^p\;dxdt \\
            &+C(p) \frac{1}{r^2}\int_{T_0}^{T} \int  a[u] u^p \eta_R^q  \;dxdt + {{C\tau}}\\
          \le & C(R,p) \left(\frac{1}{T_1-T_0} + 1\right)\int_{T_0}^{T} \int a[u]\eta_R^q u^p\;dxdt + {{C\tau}}.    
      \end{align*}
      Taking the power $\frac{1}{p}$ on both sides one gets 
      \begin{align*}
     \left ( \int_{T_1}^T \int a[u] \eta_r^q u^{p \frac{q}{2}} \;dxdt \right )^{\frac{1}{p}\frac{2}{q}} \le & \left( C(R,p) \left(\frac{1}{T_1-T_0} + 1\right)\right)^{\frac{1}{p}} \left(\int_{T_0}^{T} \int a[u]\eta_R^q u^p\;dxdt \right)^{\frac{1}{p}} +  {{C\tau^{\frac{1}{p}}}}.
      \end{align*}
	The last inequality suggests an iteration, taking successively exponents values $p \frac{q}{2}$, $p \left(\frac{q}{2}\right)^2$, $p \left(\frac{q}{2}\right)^3$, ..., $p \left(\frac{q}{2}\right)^n$ and a sequence of times are radii 
	$$ T_n = \frac{T}{4} \left( 2-\frac{1}{2^n}\right),\quad R_n = \frac{R}{2} \left( 1+\frac{1}{2^n}\right).$$
	For 
	$$
	E_n[u^{(\tau)}] := \left( \int_{T_n}^T \int a[u^{(\tau)}] \eta_n^q {u^{(\tau)}}^{p \left(\frac{q}{2}\right)^n} \;dxdt \right )^{\frac{1}{p}\left(\frac{2}{q}\right)^n}
$$
 one gets recursively 
 \begin{align*}
 E_{n+1} & \le 2^{n\frac{1}{p}\left(\frac{2}{q}\right)^n}\left( C(R,p) \left(\frac{1}{T} + 1\right)\right)^{\frac{1}{p}\left(\frac{2}{q}\right)^n}E_n  +  {{C\tau^{\frac{1}{p}\left(\frac{2}{q}\right)^n}}}\\
&  \le 2^{\frac{1}{p} \sum_1^n j\left(\frac{2}{q}\right)^j}\left( C(R,p) \left(\frac{1}{T} + 1\right)\right)^{\frac{1}{p}\sum_1^n \left(\frac{2}{q}\right)^j}E_0 + {{  C \sum_1^n\tau^{\frac{1}{p}\left(\frac{2}{q}\right)^i}}},
 \end{align*}
 with 
 $$
 E_0 = \left( \int_{T/4}^T \int_{B_R(0)}  a[u^{(\tau)}] {u^{(\tau)}}^{p} \;dxdt \right )^{\frac{1}{p}}.
 $$
 We first recall that 
$$
\| a[u^{(\tau)}]\|_{L^\infty(0,T,L^q(B_R(0)))} \le C,
$$
with $C$ independent on $\tau$ for each  $q\le 3$, which implies, together with H\"older's inequality, 
 \begin{align*}
      \int_{0}^{T}\int_{B_R(0)}  {u^{(\tau)}}^pa[u^{(\tau)}]\;  dxdt 
       \leq &   \int_{0}^{T} \left( \int_{B_R(0)}  a^q[u^{(\tau)}]\;  dx\right)^{\frac{1}{q}} \left( \int_{B_R(0)}  {u^{(\tau)}}^{pq'}\;  dx\right)^{\frac{1}{q'}} \;dt\\
       \le &  C(R) \int_{0}^{T}  \left( \int_{B_R(0)}  {u^{(\tau)}}^{pq'}\;  dx\right)^{\frac{1}{q'}}  \;dt 
    \end{align*}	  		
with $
q'\ge \frac{3}{2}$. Choosing $p$ such that 
\begin{align*}
pq' = \frac{5}{3}
\end{align*}
and applying (\ref{June20.u53}) we get 
 \begin{align*}
      \int_{0}^{T}\int_{B_R(0)}  {u^{(\tau)}}^pa[u^{(\tau)}]\;  dxdt  \leq C(R),
    \end{align*}	
    for some $1<p\le 10/9$ and $C(R)$ solely dependent on $R$ (and not on $\tau$). Taking into account that for $2<q<10/3$ and $1<p\le 10/9$
    \begin{align*}
    2^{\frac{1}{p} \sum_1^n j\left(\frac{2}{q}\right)^j}\left( C(R,p) \left(\frac{1}{T} + 1\right)\right)^{\frac{1}{p}\sum_1^n \left(\frac{2}{q}\right)^j} & \le 2^{\frac{1}{p} \sum_1^\infty j\left(\frac{2}{q}\right)^j}\left( C(R,p) \left(\frac{1}{T} + 1\right)\right)^{\frac{1}{p}\sum_1^\infty \left(\frac{2}{q}\right)^j}\\
    & \le C(R,p) \left(\frac{1}{T} + 1\right)^{\alpha},\quad \textrm{with} \quad \alpha > \frac{9}{4}
    \end{align*}
    we get that 
    $$
    E_{n+1}[u^{(\tau)}] \le C(R) \left(\frac{1}{T} + 1\right)^{\alpha} + {{  C \sum_1^n\tau^{\frac{1}{p}\left(\frac{2}{q}\right)^i}}}. 
    $$
    Taking the $\liminf$ for $\tau \to 0$ on both side of the last inequality and using pointwise convergences \eqref{utau.aecnv} and \eqref{conv_punt_a} and Fatou's lemma, for each $n\ge 0$ one gets 
     $$
    E_{n+1}[u] \le C(R) \left(\frac{1}{T} + 1\right)^{\alpha}.
    $$
 Since 
    $$
    E_n \ge \left(\inf_{B_{R/2}(0) \times (T/2,T)} a[u]\right)^{\frac{1}{p}\left(\frac{2}{q}\right)^n} \left( \int_{T/2}^T \int_{B_{R/2}(0)} u^{p \left(\frac{q}{2}\right)^n} \;dxdt \right )^{\frac{1}{p}\left(\frac{2}{q}\right)^n}
    $$
 the limit $n\to +\infty$   yields 
 $$
 \|u\|_{L^\infty( B_{R/2}(0) \times (T/2,T))} \le  C(R) \left(\frac{1}{T} + 1\right)^{\alpha},\quad \alpha > \frac{9}{4}.
 $$
 This finishes the proof of Theorem \ref{thr.reg}.

 
\section*{Appendix}

\begin{lemma}[Another sufficient condition for the $\eps-$Poincar\'e inequality]\label{lem.cnd.epsPoi}
Let $u : \R^{3}\times [0,\infty)\to [0,\infty)$ be a solution to \eqref{landau} with the properties stated in Thr.~\ref{thr.ex}. Moreover
let \eqref{cnd.epsPoi} hold. Then the $\eps-$Poincar\'e inequality \eqref{eqn:epsilon_Poincare_inequality strongerII}
is fulfilled for all functions $\phi : \R^{3}\to\R$ such that the right-hand side of \eqref{eqn:epsilon_Poincare_inequality strongerII}
is finite.
\end{lemma}
\begin{proof}
From the Poisson equation $-\Delta a = u$ it follows
\begin{align*}
& \int_{\R^{3}}u\phi^{2}dx = -\int_{\R^{3}}\phi^{2}\Delta a dx = 2\int_{\R^{3}}\phi\na a\cdot\na \phi dx.
\end{align*}
Applying Young's inequality leads to
\begin{align}\label{June27.1}
& \int_{\R^{3}}u\phi^{2}dx \leq \eps\int_{\R^{3}}|\na \phi |^{2}\gamma dx + \frac{1}{\eps}\int_{\R^{3}}|\phi\na a|^{2}\gamma^{-1}dx .
\end{align}
Let us consider the second integral on the right-hand side of \eqref{June27.1}. Let $1<p<3$.
H\"older inequality allows us to write
\begin{align*}
& \frac{1}{\eps}\int_{\R^{3}}|\phi\na a|^{2}\gamma^{-1}dx\leq 
\frac{1}{\eps}\left( \int_{\R^{3}} |\gamma^{-1}\na a|^{2p/(p-1)} dx \right)^{1-1/p}
\left( \int_{\R^{3}} \gamma^{p}|\phi|^{2p} dx \right)^{1/p}.
\end{align*}
Assumption \eqref{cnd.epsPoi} implies that a suitable $p<3$ exists such that
\begin{align}\label{June27.1b}
& \frac{1}{\eps}\int_{\R^{3}}|\phi\na a|^{2}\gamma^{-1}dx\leq 
\frac{C_{T}}{\eps}\left( \int_{\R^{3}} \gamma^{p}|\phi|^{2p} dx \right)^{1/p} = \frac{C_{T}}{\eps}\|\gamma^{1/2}\phi\|_{L^{2p}(\R^{3})}^{2}\qquad 0<t<T,
\end{align}
for some constant $C_{T}>0$ depending on $T$. Gagliardo-Nirenberg inequality implies that
\begin{align}\label{June27.GN}
& \|\gamma^{1/2}\phi\|_{L^{2p}(\R^{3})}\leq C\|\na(\gamma^{1/2}\phi)\|_{L^{2}(\R^{3})}^{\eta}\|\gamma^{1/2}\phi\|_{L^{2}(\R^{3})}^{1-\eta},
\end{align}
for some $\eta\in (0,1)$. We point out that $\eta<1$ because $p<3$.
Putting \eqref{June27.1b}, \eqref{June27.GN} together leads to
\begin{align*}
& \frac{1}{\eps}\int_{\R^{3}}|\phi\na a|^{2}\gamma^{-1}dx\leq 
\frac{C_{T}}{\eps}\|\na(\gamma^{1/2}\phi)\|_{L^{2}(\R^{3})}^{2\eta}\|\gamma^{1/2}\phi\|_{L^{2}(\R^{3})}^{2(1-\eta)}\qquad 0<t<T.
\end{align*}
From Young inequality it follows
\begin{align*}
& \frac{1}{\eps}\int_{\R^{3}}|\phi\na a|^{2}\gamma^{-1}dx\leq 
\eps\|\na(\gamma^{1/2}\phi)\|_{L^{2}(\R^{3})}^{2} + C_{T,\eps}\|\gamma^{1/2}\phi\|_{L^{2}(\R^{3})}^{2}\qquad 0<t<T.
\end{align*}
However, since $|\na(\gamma^{1/2}\phi)|\leq \gamma^{1/2}|\na\phi| + \frac{1}{2}\gamma^{3/2}|\phi|\leq  \gamma^{1/2}|\na\phi| + C|\gamma^{1/2}\phi|$, we obtain
\begin{align*}
& \frac{1}{\eps}\int_{\R^{3}}|\phi\na a|^{2}\gamma^{-1}dx\leq 
\eps\int_{\R^{3}}|\na \phi |^{2}\gamma dx  + C_{T,\eps}\|\gamma^{1/2}\phi\|_{L^{2}(\R^{3})}^{2}\qquad 0<t<T.
\end{align*}
Plugging the above inequality inside \eqref{June27.1} and noticing that $\gamma\leq 1$ yields 
\begin{align}\label{June27.2}
& \int_{\R^{3}}u\phi^{2}dx \leq 2\eps\int_{\R^{3}}|\na \phi |^{2}\gamma dx + C_{T,\eps}\int_{\R^{3}}|\phi|^{2}dx\qquad 0<t<T.
\end{align}
However, it is known that $a[u](x)\geq C\gamma(x)$ for $x\in\R^{3}$, with $C$ being a positive
constant which depends only on the mass and second moment of $u$. Since we know that 
$\sup_{0<t<T}\int_{\R^{3}}(1+|x|^{2})u(x,t)dx<\infty$, from \eqref{June27.2} inequality
\eqref{eqn:epsilon_Poincare_inequality strongerII} follows. This finishes the proof.
\end{proof}

\end{document}